\documentclass[12pt, a4paper, reqno]{amsart}
\usepackage{graphicx}
\usepackage[dvipsnames]{xcolor}
\definecolor{plum}{rgb}{0.44, 0.11, 0.11}%persianplum
\definecolor{red}{rgb}{0.8, 0.2, 0.2}%persianred
\definecolor{green}{rgb}{0.07, 0.21, 0.14}%phthalogreen
\definecolor{blue}{rgb}{0.0, 0.33, 0.71}%mediumtealblue

\usepackage{hyperref}
\hypersetup{colorlinks=true,linkcolor=red,citecolor=blue,}
\usepackage{amsmath,amsfonts,amssymb,amsthm,amscd}
\usepackage{tikz}
\usepackage{nicematrix}
\usetikzlibrary{arrows.meta, shapes, trees, calc}
\newcommand{\e}{\varepsilon}
\newcommand{\f}{\varphi}
\newcommand{\Fc}{{\mathcal F}}

\newcommand{\Lc}{{\mathcal L}}
\newcommand{\Nc}{{\mathcal N}}
\newcommand{\N}{{\mathbb N}}
\newcommand{\Z}{{\mathbb Z}}
\newcommand{\R}{{\mathbb R}}
\newcommand{\C}{{\mathbb C}}
\newcommand{\br}[1]{\left\langle #1 \right\rangle}
\newcommand{\wh}{\widehat}
\newcommand{\wt}{\widetilde}
\newcommand{\wb}{\overline}

\DeclareMathOperator{\im}{{\rm im}}
\DeclareMathOperator{\supp}{{\rm supp}}
\DeclareMathOperator{\id}{{\rm id}}

\DeclareMathOperator{\Cay}{{\rm Cay}}

\DeclareMathOperator{\Prob}{{\bf P}}
\DeclareMathOperator{\E}{{\bf E}}

\DeclareMathOperator{\Hess}{{\rm Hess}}
\newcommand{\1}{{\bf 1}}
\newcommand{\mb}[1]{\boldsymbol{#1}}
\newcommand{\dr}[2]{\chi_#1(#2)} %drift; non-reversibility
\newcommand{\rad}{{\rm rad}} %spectral radius
\newcommand{\diag}{{\rm diag}}
\newcommand{\TV}{{\rm TV}}
\newcommand{\uth}{{\underline{\vartheta}}}
\numberwithin{equation}{section}
\newtheorem{theorem}{Theorem}[section]
\newtheorem{proposition}[theorem]{Proposition}
\newtheorem{lemma}[theorem]{Lemma}
\newtheorem{fact}[theorem]{Fact}
\newtheorem{facts}[theorem]{Facts}
\newtheorem{corollary}[theorem]{Corollary}

\theoremstyle{definition}

\newtheorem{remark}[theorem]{Remark}
\newtheorem{example}[theorem]{Example}
\begin{document}

\title[Noise sensitivity on virtually abelian groups]
      {Noise sensitivity on virtually abelian groups}
\author{J\'er\'emie Brieussel}
\author{Ryokichi Tanaka}
\date{\today}

\newcommand{\jb}{
\begin{flushleft}
\textsc{J\'er\'emie Brieussel}\\
Institut Montpellierain Alexander Grothendieck,\\
Universit\'e de Montpellier, CNRS, Montpellier, France\\
\textit{E-mail address}: {\tt jeremie.brieussel@umontpellier.fr}
\end{flushleft}
}

\newcommand{\rt}{
\begin{flushleft}
\textsc{Ryokichi Tanaka}\\
Department of Mathematics,\\
Kyoto University,
Kyoto, Japan\\
\textit{E-mail address}: {\tt rtanaka@math.kyoto-u.ac.jp}
\end{flushleft}
}

\maketitle

\begin{abstract}
We show that aperiodic random walks with finite second moment on virtually abelian groups are noise sensitive in total variation if and only if the group admits no nonzero homomorphism onto the infinite cyclic group.
\end{abstract}

\section{Introduction}\label{sec:intro}

Let $\Gamma$ be a countable group and $\mu$
be a probability measure on the group.
The $\mu$-\emph{random walk} 
$\{w_n\}_{n \in \N}$
starting from the identity $\id$ is defined by $w_n:=\gamma_1\cdots \gamma_n$ and $w_0:=\id$
for an independent, identically distributed sequence $\gamma_1, \gamma_2, \dots$ with the law $\mu$.
The distribution of $w_n$ coincides with the $n$-fold convolution $\mu_n:=\mu^{\ast n}$.
The \emph{noise sensitivity problem} on $\mu$-random walks asks the following.
For each fixed real $\rho \in (0, 1)$,
let $w_n^\rho:=\gamma^\rho_1\cdots \gamma^\rho_n$,
where $\gamma_i^\rho$ is resampled with the same law $\mu$ with probability $\rho$,
or is retained as $\gamma_i$ with probability $1-\rho$,
independently in every step.
Is the pair $w_n$ and $w_n^\rho$ nearly independent or highly correlated after a large enough time $n$?

More precisely, for each $\rho \in [0, 1]$,
let
\[
\pi^\rho:=\rho (\mu\times \mu)+(1-\rho)\mu_\diag \quad \text{on $\Gamma \times \Gamma$},
\]
where $\mu\times \mu$ denotes the product measure and $\mu_\diag((\gamma, \gamma')):=\mu(\gamma)$ if $\gamma=\gamma'$ and $0$ if else.
For a $\pi^\rho$-random walk $\{\mb{w}_n\}_{n \in \N}$ starting from $\id$ on $\Gamma\times \Gamma$,
the distribution $\pi^\rho_n$ coincides with that of the pair $(w_n, w_n^\rho)$.
We say that a $\mu$-random walk on $\Gamma$ is \emph{noise sensitive} in total variation (or, equivalently in $\ell^1$) if 
\[
\lim_{n \to \infty}\|\pi^\rho_n-\mu_n\times\mu_n\|_\TV=0
\quad \text{for each $\rho \in (0, 1]$}.
\]
Recall that the total variation distance between two probability measures $\nu_i$ for $i=1, 2$ on $\Gamma \times \Gamma$ is given by $\|\nu_1-\nu_2\|_\TV:=\max_{A \subset \Gamma\times \Gamma}|\nu_1(A)-\nu_2(A)|=(1/2)\|\nu_1-\nu_2\|_1$.

Roughly speaking,
if a $\mu$-random walk on $\Gamma$ is noise sensitive,
then even small $\rho$-portion of resampling in increments causes a significant effect,
producing an independent copy of the walk eventually.
The noise sensitivity problem on random walks on groups was introduced by Benjamini and the first named author in \cite{bb} (see also a related discussion in \cite[Section 3.3.4]{kalai}).
In that paper, the authors discuss other variants of the notion in terms of entropy and of different measurements for the distance.
In the present paper, we focus on noise sensitivity in total variation,
and thus we call it noise sensitivity for brevity.
A basic question one can ask: What groups admit noise sensitive random walks?

Both examples and non-examples of noise sensitive random walks have been discovered. 
On the one hand, there are two obstructions to noise sensitivity that apply to large families of groups~\cite[Theorem 1.1]{bb}. 
First, if $\Gamma$ admits a nonzero homomorphism onto $\Z$, then a $\mu$-random walk is not noise sensitive.
This follows from an inspection on the central limit theorem. 
Second, if $(\Gamma, \mu)$ is non-Liouville, i.e.\ $\Gamma$ admits a non-constant bounded $\mu$-harmonic functions,
then a $\mu$-random walk is not noise sensitive.
This follows by using the representations of bounded $\mu$-harmonic functions in terms of $\mu$-stationary distributions.
So far, these two results are the only known obstructions for noise sensitivity. 

On the other hand, there are some examples of noise-sensitive random walks. Besides finite groups which are obviously noise sensitive~\cite[Proposition 5.1]{bb}, the infinite dihedral group, and more generally every affine Weyl group admits some noise sensitive random walk  by~\cite[Theorem 1.1]{t-ns} (see also \cite[Theorem 1.4]{bb} for the dihedral group).
Note that these groups admit finite index abelian subgroups but do not admit nonzero homomorphisms onto the infinite cyclic group $\Z$ since they have a set of generators consisting of finite order elements.

The present paper deals with the class of finitely generated virtually abelian groups, i.e.\ groups which admit finite index abelian subgroups. 
Random walks on such groups are Liouville (see, e.g., \cite{fhtvf}).
We show that, in this class,
the existence of nonzero homomorphisms onto $\Z$ is the only obstruction to noise sensitivity for {\it aperiodic} random walks. This provides new examples of noise sensitive random walks, see Examples~\ref{ex:di} and~\ref{ex:tri} below.

Recall that $\mu$ is \emph{aperiodic} if there exist at least two positive probability random walk paths from $\id$ to $\id$ such that their lengths are mutually prime, e.g., if $\mu(\id)>0$
(cf.\ Section \ref{sec:char}). Also recall that given a word metric $|\cdot|$ in $\Gamma$, a probability measure $\mu$ has \emph{finite second moment} if
$\sum_{\gamma \in \Gamma}|\gamma|^2\mu(\gamma)<\infty$.
This condition is independent of the choice of word metrics since arbitrary such metrics are bi-Lipschitz equivalent.

\begin{theorem}\label{thm:intro}
Let $\Gamma$ be a finitely generated virtually abelian group and $\mu$ be an aperiodic probability measure on $\Gamma$ with finite second moment such that the support generates the group as a semigroup.
Then the $\mu$-random walk on $\Gamma$ is noise sensitive in total variation, i.e.\ 
\[
\left\|\pi^\rho_n-\mu_n\times \mu_n\right\|_\TV \to 0 \quad \text{as $n\to \infty$},
\]
for all $\rho \in (0, 1]$
if and only if $\Gamma$ does not admit nonzero homomorphisms onto $\Z$.
\end{theorem}

This Theorem \ref{thm:intro} generalizes the known result on affine Weyl groups with $\mu$ supported on the canonical set of generators and the identity, established by the second named author \cite[Theorem 1.1]{t-ns}.
%Note that $(\Gamma, \mu)$ is Liouville in Theorem \ref{thm:intro}.
If the group $\Gamma$ admits a nonzero homomorphism onto $\Z$,
then the asymptotic behavior of the total variation distance depends on the parameter $\rho$.

\begin{theorem}\label{thm:nonzerohom_intro}
Let $\Gamma$ be a finitely generated virtually abelian group and $\mu$ be a finite second moment probability measure on $\Gamma$ such that the support generates the group as a semigroup.
Then
\[
\lim_{\rho \to 1}\limsup_{n \to \infty}\left\|\pi^\rho_n-\mu_n\times \mu_n\right\|_\TV=0. 
\]
If moreover $\Gamma$ admits a nonzero homomorphism onto $\Z$,
then
\[%\quad \text{and} \quad
\lim_{\rho \to 0}\liminf_{n \to \infty}\left\|\pi^\rho_n-\mu_n\times \mu_n\right\|_\TV=1,
\]
in particular, the $\mu$-random walk on $\Gamma$ is not noise sensitive in total variation.
\end{theorem}

Note that in the above Theorem \ref{thm:nonzerohom_intro}, the probability measure $\mu$ is not necessarily aperiodic.
%In fact, we will establish strengthened versions of Theorems \ref{thm:intro} and \ref{thm:nonzerohom_intro} under the finite second moment condition in Theorems \ref{thm:zerohom} and \ref{thm:nonzerohom} respectively in Section \ref{sec:ns}.
In the course of proving Theorem \ref{thm:intro},
we observe the following more general ``decoupling'' phenomenon on a product group $\Gamma_1 \times \Gamma_2$, where $\Gamma_1$ and $\Gamma_2$ may not be isomorphic.
Two coupled random walks on virtually abelian groups $\Gamma_1$ and $\Gamma_2$, potentially highly correlated though, eventually behave as two independent random walks if one of the factors admits no nonzero homomorphism onto $\Z$.

\begin{theorem}\label{thm:zerohom_intro}
For $i=1, 2$, let $\Gamma_i$ be a finitely generated virtually abelian group,
and $\nu$ be an aperiodic probability measure with finite second moment on $\Gamma_1\times \Gamma_2$ such that the support of $\nu$ generates $\Gamma_1\times \Gamma_2$ as a semigroup and $\nu$ has the marginal $\mu^{(i)}$ on $\Gamma_i$.
If $\Gamma_1$ admits no nonzero homomorphism onto $\Z$,
then 
\[
\big\|\nu_n-\mu^{(1)}_n\times \mu^{(2)}_n\big\|_\TV \to 0 \quad \text{as $n \to \infty$}.
\]
\end{theorem}

We apply this result to $\pi^\rho$ for $\rho\in (0, 1]$ and deduce noise sensitivity for $\mu$-random walks on groups $\Gamma$ with no nonzero homomorphisms to $\Z$.

Before we describe the outlines of the proofs, we discuss some explicit examples.

\begin{example}\label{ex:di}
Let us consider the infinite dihedral group 
\[
D_\infty:=\br{s_1, s_2 \mid s_1^2=s_2^2=\id}.
\]
The group $D_\infty$ is isomorphic to the semi-direct product $\Z\rtimes (\Z/2)$. It does not admit nonzero homomorphisms onto $\Z$ since it is generated by two involutions $s_1$ and $s_2$.
Let $\mu_{\rm LSRW}$ be the uniform distribution on $\{s_1, s_2, \id\}$.
Then, the $\mu_{\rm LSRW}$-random walk (``lazy simple random walk'') on $D_\infty$ is noise sensitive by Theorem \ref{thm:intro}.
The noise sensitivity of this special random walk on $D_\infty$ has already been proven by a coupling method in \cite[Section 6]{bb}.
Let us provide another example.
Let $\mu_{\rm ape}$ be the uniform distribution on $\{s_1, s_2, s_1s_2, s_2s_1\}$.
Then, $\mu_{\rm ape}$ is aperiodic.
This follows since the $\mu_{\rm ape}$-random walk starting from $\id$ returns to $\id$ with positive probability at times $2$ and $3$.
Theorem \ref{thm:intro} implies that the $\mu_{\rm ape}$-random walk on $D_\infty$ is noise sensitive.
\end{example}

In the above Example \ref{ex:di}, 
if $\mu_{\rm SRW}$ denotes the uniform distribution on $\{s_1, s_2\}$,
then the $\mu_{\rm SRW}$-random walk (``simple random walk'') on $D_\infty$ is not noise sensitive.
Indeed, the support of $(\mu_{\rm SRW})^{\ast 2}$ is included in $\Z$ under the identification $D_\infty=\Z\rtimes (\Z/2)$, and $(\mu_{\rm SRW})^{\ast 2}$ defines the lazy simple random walk on $\Z$ with the holding probability $1/2$. 
Furthermore, for each $\rho \in (0, 1]$,
and for the probability measure $\pi^\rho$, the support of $(\pi^\rho)^{\ast 2}$ is included in the subgroup $\Z^2$ of $D_\infty^2$ under the identification $D_\infty^2=\Z^2\rtimes (\Z/2)^2$.
If the $\mu_{\rm SRW}$-random walk on $D_\infty$ were noise sensitive,
then we would have 
\[
\|\pi^\rho_{2n}-(\mu_{\rm SRW})_{2n}\times (\mu_{\rm SRW})_{2n}\|_\TV \to 0 \quad \text{
as $n \to \infty$}.
\]
However, this does not hold for the $(\pi^\rho)^{\ast 2}$-random walk on $\Z^2$: the total variation distance does not tend to $0$ for $\rho$ close enough to $0$.
Indeed, the $(\pi^\rho)^{\ast 2}$-random walk on $\Z$ is not noise sensitive, which follows from the proof of Theorem \ref{thm:nonzerohom_intro}.
% (cf.\ Remark \ref{rem:continuity}).
This in particular indicates that the property of being noise sensitive does depend on the random walk (not solely on the group), as already pointed out in \cite[Theorem 1.4]{bb}.

\begin{figure}[h]
\centering
%%begin%%tikz%%%%%%%%%%%%%%%%%%%%%%%%%%%%%%%%%%%%%%%%%%%

\scalebox{0.8}[0.8]{
\begin{tikzpicture}[
dot/.style={draw, coordinate},
ver/.style={draw, circle, scale=0.4, fill=black},
]
%nodes

\node[dot] at (-8,{-4*sqrt(3)}) (-8,{-4*sqrt(3)}) {};
\node[dot] at (-4,{-4*sqrt(3)}) (-4,{-4*sqrt(3)}) {};
\node[dot] at (0,{-4*sqrt(3)}) (0,{-4*sqrt(3)}) {};
\node[dot] at (4,{-4*sqrt(3)}) (4,{-4*sqrt(3)}) {};
\node[dot] at (8,{-4*sqrt(3)}) (8,{-4*sqrt(3)}) {};

\node[dot] at (-6, {-2*sqrt(3)}) (-6, {-2*sqrt(3)}) {};
\node[dot] at (-2, {-2*sqrt(3)}) (-2, {-2*sqrt(3)}) {};
\node[dot] at (2, {-2*sqrt(3)}) (2, {-2*sqrt(3)}) {};
\node[dot] at (6, {-2*sqrt(3)}) (6, {-2*sqrt(3)}) {};
\node[dot] at (10, {-2*sqrt(3)}) (10, {-2*sqrt(3)}) {};

\node[dot] at (-8,0) (-8,0) {};
\node[dot] at (-4,0) (-4,0) {};
\node[dot] at (0,0) (0,0) {};
\node[dot] at (4,0) (4,0) {};
\node[dot] at (8,0) (8,0) {};

\node[dot] at (-6, {2*sqrt(3)}) (-6, {2*sqrt(3)}) {};
\node[dot] at (-2, {2*sqrt(3)}) (-2, {2*sqrt(3)}) {};
\node[dot] at (2, {2*sqrt(3)}) (2, {2*sqrt(3)}) {};
\node[dot] at (6, {2*sqrt(3)}) (6, {2*sqrt(3)}) {};
\node[dot] at (10, {2*sqrt(3)}) (10, {2*sqrt(3)}) {};

\node[dot] at (-8,{4*sqrt(3)}) (-8,{4*sqrt(3)}) {};
\node[dot] at (-4,{4*sqrt(3)}) (-4,{4*sqrt(3)}) {};
\node[dot] at (0,{4*sqrt(3)}) (0,{4*sqrt(3)}) {};
\node[dot] at (4,{4*sqrt(3)}) (4,{4*sqrt(3)}) {};
\node[dot] at (8,{4*sqrt(3)}) (8,{4*sqrt(3)}) {};

\draw[dotted, fill=gray!20] (4,0)--(2,{2*sqrt(3)})--(-2,{2*sqrt(3)})--(-4,0)--(-2,{-2*sqrt(3)})--(2,{-2*sqrt(3)})--cycle;
\draw[dotted, fill=gray!50] (2,{2*sqrt(3)})--(-2,{2*sqrt(3)})--(0,0)--cycle;
\draw[dotted, fill=gray!100] (0,0)--(4,0)--(2,{2*sqrt(3)})--cycle;

%paths
%horizontal

\draw[dotted] (-8,{-4*sqrt(3)})--(-4,{-4*sqrt(3)});
\draw[dotted] (-4,{-4*sqrt(3)})--(0,{-4*sqrt(3)});
\draw[dotted] (0,{-4*sqrt(3)})--(4,{-4*sqrt(3)});
\draw[dotted] (4,{-4*sqrt(3)})--(8,{-4*sqrt(3)});

\draw[dotted] (-8, {-2*sqrt(3)})--(-6, {-2*sqrt(3)});
\draw[dotted] (-6, {-2*sqrt(3)})--(-2, {-2*sqrt(3)});
\draw[dotted] (-2, {-2*sqrt(3)})--(2, {-2*sqrt(3)});
\draw[dotted] (2, {-2*sqrt(3)})--(6, {-2*sqrt(3)});
\draw[dotted] (6, {-2*sqrt(3)})--(8, {-2*sqrt(3)});

\draw[dotted] (-8,0)--(-4,0);
\draw[dotted] (-4,0)--(0,0);
\draw[dotted] (0,0)--(4,0);
\draw[dotted] (4,0)--(8,0);

\draw[dotted] (-8, {2*sqrt(3)})--(-6, {2*sqrt(3)});
\draw[dotted] (-6, {2*sqrt(3)})--(-2, {2*sqrt(3)});
\draw[dotted] (-2, {2*sqrt(3)})--(2, {2*sqrt(3)});
\draw[dotted] (2, {2*sqrt(3)})--(6, {2*sqrt(3)});
\draw[dotted] (6, {2*sqrt(3)})--(8, {2*sqrt(3)});

\draw[dotted] (-8,{4*sqrt(3)})--(-4,{4*sqrt(3)});
\draw[dotted] (-4,{4*sqrt(3)})--(0,{4*sqrt(3)});
\draw[dotted] (0,{4*sqrt(3)})--(4,{4*sqrt(3)});
\draw[dotted] (4,{4*sqrt(3)})--(8,{4*sqrt(3)});

%cross

\draw[dotted] (-8,{-4*sqrt(3)})--(-6, {-2*sqrt(3)});
\draw[dotted] (-4,{-4*sqrt(3)})--(-6, {-2*sqrt(3)});
\draw[dotted] (-4,{-4*sqrt(3)})--(-2, {-2*sqrt(3)});
\draw[dotted] (0,{-4*sqrt(3)})--(-2, {-2*sqrt(3)});
\draw[dotted] (0,{-4*sqrt(3)})--(2, {-2*sqrt(3)});
\draw[dotted] (4,{-4*sqrt(3)})--(2, {-2*sqrt(3)});
\draw[dotted] (4,{-4*sqrt(3)})--(6, {-2*sqrt(3)});
\draw[dotted] (8,{-4*sqrt(3)})--(6, {-2*sqrt(3)});

\draw[dotted] (-8,0)--(-6, {-2*sqrt(3)});
\draw[dotted] (-4,0)--(-6, {-2*sqrt(3)});
\draw[dotted] (-4,0)--(-2, {-2*sqrt(3)});
\draw[dotted] (0,0)--(-2, {-2*sqrt(3)});
\draw[dotted] (0,0)--(2, {-2*sqrt(3)});
\draw[dotted] (4,0)--(2, {-2*sqrt(3)});
\draw[dotted] (4,0)--(6, {-2*sqrt(3)});
\draw[dotted] (8,0)--(6, {-2*sqrt(3)});

\draw[dotted] (-8,0)--(-6, {2*sqrt(3)});
\draw[dotted] (-4,0)--(-6, {2*sqrt(3)});
\draw[dotted] (-4,0)--(-2, {2*sqrt(3)});
\draw[dotted] (0,0)--(-2, {2*sqrt(3)});
\draw[dotted] (0,0)--(2, {2*sqrt(3)});
\draw[dotted] (4,0)--(2, {2*sqrt(3)});
\draw[dotted] (4,0)--(6, {2*sqrt(3)});
\draw[dotted] (8,0)--(6, {2*sqrt(3)});

\draw[dotted] (-8,{4*sqrt(3)})--(-6, {2*sqrt(3)});
\draw[dotted] (-4,{4*sqrt(3)})--(-6, {2*sqrt(3)});
\draw[dotted] (-4,{4*sqrt(3)})--(-2, {2*sqrt(3)});
\draw[dotted] (0,{4*sqrt(3)})--(-2, {2*sqrt(3)});
\draw[dotted] (0,{4*sqrt(3)})--(2, {2*sqrt(3)});
\draw[dotted] (4,{4*sqrt(3)})--(2, {2*sqrt(3)});
\draw[dotted] (4,{4*sqrt(3)})--(6, {2*sqrt(3)});
\draw[dotted] (8,{4*sqrt(3)})--(6, {2*sqrt(3)});

%Cayley graph
%edges

%horizontal

\draw[latex-, color = plum, line width=0.8mm] (-8, {(2/3)*sqrt(3)})--(-6, {(2/3)*sqrt(3)});
\draw[latex-, color = green, line width=0.8mm] (-4, {(2/3)*sqrt(3)})--(-2, {(2/3)*sqrt(3)});
\draw[color = green, line width=0.8mm] (-6, {(2/3)*sqrt(3)})--(-2, {(2/3)*sqrt(3)});
\draw[latex-, color = red, line width=0.8mm] (0, {(2/3)*sqrt(3)})--(2, {(2/3)*sqrt(3)});
\draw[color = red, line width=0.8mm] (-2, {(2/3)*sqrt(3)})--(2, {(2/3)*sqrt(3)}) node[pos=0.5, below=0.1mm] {$s_1$};
\draw[latex-, color = plum, line width=0.8mm] (4, {(2/3)*sqrt(3)})--(6, {(2/3)*sqrt(3)});
\draw[color = plum, line width=0.8mm] (2, {(2/3)*sqrt(3)})--(6, {(2/3)*sqrt(3)}) node[pos=0.5, below=0.1mm] {$s_2$};
\draw[color = green, line width=0.8mm] (6, {(2/3)*sqrt(3)})--(8, {(2/3)*sqrt(3)});

\draw[latex-, color = red, line width=0.8mm] (-6, {(8/3)*sqrt(3)})--(-4, {(8/3)*sqrt(3)});
\draw[color = red, line width=0.8mm] (-8, {(8/3)*sqrt(3)})--(-4, {(8/3)*sqrt(3)});
\draw[latex-, color = plum, line width=0.8mm] (-2, {(8/3)*sqrt(3)})--(0, {(8/3)*sqrt(3)});
\draw[color = plum, line width=0.8mm] (-4, {(8/3)*sqrt(3)})--(0, {(8/3)*sqrt(3)});
\draw[latex-, color = green, line width=0.8mm] (2, {(8/3)*sqrt(3)})--(4, {(8/3)*sqrt(3)});
\draw[color = green, line width=0.8mm] (0, {(8/3)*sqrt(3)})--(4, {(8/3)*sqrt(3)})node[pos=0.5, below=1.0mm] {$s_3$};
\draw[latex-, color = red, line width=0.8mm] (6, {(8/3)*sqrt(3)})--(8, {(8/3)*sqrt(3)});
\draw[color = red, line width=0.8mm] (4, {(8/3)*sqrt(3)})--(8, {(8/3)*sqrt(3)});

\draw[latex-, color = red, line width=0.8mm] (-6, {-(4/3)*sqrt(3)})--(-4, {-(4/3)*sqrt(3)});
\draw[color = red, line width=0.8mm] (-8, {-(4/3)*sqrt(3)})--(-4, {-(4/3)*sqrt(3)});
\draw[latex-, color = plum, line width=0.8mm] (-2, {-(4/3)*sqrt(3)})--(0, {-(4/3)*sqrt(3)});
\draw[color = plum, line width=0.8mm] (-4, {-(4/3)*sqrt(3)})--(0, {-(4/3)*sqrt(3)});
\draw[latex-, color = green, line width=0.8mm] (2, {-(4/3)*sqrt(3)})--(4, {-(4/3)*sqrt(3)});
\draw[color = green, line width=0.8mm] (0, {-(4/3)*sqrt(3)})--(4, {-(4/3)*sqrt(3)});
\draw[latex-, color = red, line width=0.8mm] (6, {-(4/3)*sqrt(3)})--(8, {-(4/3)*sqrt(3)});
\draw[color = red, line width=0.8mm] (4, {-(4/3)*sqrt(3)})--(8, {-(4/3)*sqrt(3)});

\draw[latex-, color = plum, line width=0.8mm] (-8, {-(10/3)*sqrt(3)})--(-6, {-(10/3)*sqrt(3)});
\draw[color = plum, line width=0.8mm] (-8, {-(10/3)*sqrt(3)})--(-6, {-(10/3)*sqrt(3)});
\draw[latex-, color = green, line width=0.8mm] (-4, {-(10/3)*sqrt(3)})--(-2, {-(10/3)*sqrt(3)});
\draw[color = green, line width=0.8mm] (-6, {-(10/3)*sqrt(3)})--(-2, {-(10/3)*sqrt(3)});
\draw[latex-, color = red, line width=0.8mm] (0, {-(10/3)*sqrt(3)})--(2, {-(10/3)*sqrt(3)});
\draw[color = red, line width=0.8mm] (-2, {-(10/3)*sqrt(3)})--(2, {-(10/3)*sqrt(3)});
\draw[latex-, color = plum, line width=0.8mm] (4, {-(10/3)*sqrt(3)})--(6, {-(10/3)*sqrt(3)});
\draw[color = plum, line width=0.8mm] (2, {-(10/3)*sqrt(3)})--(6, {-(10/3)*sqrt(3)});
\draw[color = green, line width=0.8mm] (6, {-(10/3)*sqrt(3)})--(8, {-(10/3)*sqrt(3)});

%cross

\draw[-latex, color = plum, line width=0.8mm] (-8, {(8/3)*sqrt(3)})--(-7, {(11/3)*sqrt(3)});
\draw[color = green, line width=0.8mm] (-5, {(11/3)*sqrt(3)})--(-4, {(8/3)*sqrt(3)});
\draw[-latex, color = green, line width=0.8mm] (-4, {(8/3)*sqrt(3)})--(-3, {(11/3)*sqrt(3)});
\draw[color = red, line width=0.8mm] (-1, {(11/3)*sqrt(3)})--(0, {(8/3)*sqrt(3)});
\draw[-latex, color = red, line width=0.8mm] (0, {(8/3)*sqrt(3)})--(1, {(11/3)*sqrt(3)});
\draw[color = plum, line width=0.8mm] (3, {(11/3)*sqrt(3)})--(4, {(8/3)*sqrt(3)});
\draw[-latex, color = plum, line width=0.8mm] (4, {(8/3)*sqrt(3)})--(5, {(11/3)*sqrt(3)});
\draw[color = green, line width=0.8mm] (7, {(11/3)*sqrt(3)})--(8, {(8/3)*sqrt(3)});

\draw[-latex, color = red, line width=0.8mm] (-8, {(8/3)*sqrt(3)})--(-7, {(5/3)*sqrt(3)});
\draw[color = red, line width=0.8mm] (-8, {(8/3)*sqrt(3)})--(-6, {(2/3)*sqrt(3)});
\draw[-latex, color = red, line width=0.8mm] (-6, {(2/3)*sqrt(3)})--(-5, {(5/3)*sqrt(3)});
\draw[color = red, line width=0.8mm] (-6, {(2/3)*sqrt(3)})--(-4, {(8/3)*sqrt(3)});
\draw[-latex, color = plum, line width=0.8mm] (-4, {(8/3)*sqrt(3)})--(-3, {(5/3)*sqrt(3)});
\draw[color = plum, line width=0.8mm] (-4, {(8/3)*sqrt(3)})--(-2, {(2/3)*sqrt(3)});
\draw[-latex, color = plum, line width=0.8mm] (-2, {(2/3)*sqrt(3)})--(-1, {(5/3)*sqrt(3)});
\draw[color = plum, line width=0.8mm] (-2, {(2/3)*sqrt(3)})--(0, {(8/3)*sqrt(3)});
\draw[-latex, color = green, line width=0.8mm] (0, {(8/3)*sqrt(3)})--(1, {(5/3)*sqrt(3)});
\draw[color = green, line width=0.8mm] (0, {(8/3)*sqrt(3)})--(2, {(2/3)*sqrt(3)}) node[pos=0.6, above=3.0mm] {$s_3$};
\draw[-latex, color = green, line width=0.8mm] (2, {(2/3)*sqrt(3)})--(3, {(5/3)*sqrt(3)});
\draw[color = green, line width=0.8mm] (2, {(2/3)*sqrt(3)})--(4, {(8/3)*sqrt(3)}) node[pos=0.4, above=3.0mm] {$s_3$};
\draw[-latex, color = red, line width=0.8mm] (4, {(8/3)*sqrt(3)})--(5, {(5/3)*sqrt(3)});
\draw[color = red, line width=0.8mm] (4, {(8/3)*sqrt(3)})--(6, {(2/3)*sqrt(3)});
\draw[-latex, color = red, line width=0.8mm] (6, {(2/3)*sqrt(3)})--(7, {(5/3)*sqrt(3)});
\draw[color = red, line width=0.8mm] (6, {(2/3)*sqrt(3)})--(8, {(8/3)*sqrt(3)});

\draw[-latex, color = plum, line width=0.8mm] (-8, {-(4/3)*sqrt(3)})--(-7, {-(1/3)*sqrt(3)});
\draw[color = plum, line width=0.8mm] (-8, {-(4/3)*sqrt(3)})--(-6, {(2/3)*sqrt(3)});
\draw[-latex, color = green, line width=0.8mm] (-6, {(2/3)*sqrt(3)})--(-5, {-(1/3)*sqrt(3)});
\draw[color = green, line width=0.8mm] (-6, {(2/3)*sqrt(3)})--(-4, {-(4/3)*sqrt(3)});
\draw[-latex, color = green, line width=0.8mm] (-4, {-(4/3)*sqrt(3)})--(-3, {-(1/3)*sqrt(3)});
\draw[color = green, line width=0.8mm] (-4, {-(4/3)*sqrt(3)})--(-2, {(2/3)*sqrt(3)});
\draw[-latex, color = red, line width=0.8mm] (-2, {(2/3)*sqrt(3)})--(-1, {-(1/3)*sqrt(3)});
\draw[color = red, line width=0.8mm] (-2, {(2/3)*sqrt(3)})--(0, {-(4/3)*sqrt(3)}) node[pos=0.6, above=3.0mm] {$s_1$};
\draw[-latex, color = red, line width=0.8mm] (0, {-(4/3)*sqrt(3)})--(1, {-(1/3)*sqrt(3)});
\draw[color = red, line width=0.8mm] (0, {-(4/3)*sqrt(3)})--(2, {(2/3)*sqrt(3)})node[pos=0.4, above=3.0mm] {$s_1$};
\draw[-latex, color = plum, line width=0.8mm] (2, {(2/3)*sqrt(3)})--(3, {-(1/3)*sqrt(3)});
\draw[color = plum, line width=0.8mm] (2, {(2/3)*sqrt(3)})--(4, {-(4/3)*sqrt(3)})  node[pos=0.6, above=2.0mm] {$s_2$};
\draw[-latex, color = plum, line width=0.8mm] (4, {-(4/3)*sqrt(3)})--(5, {-(1/3)*sqrt(3)});
\draw[color = plum, line width=0.8mm] (4, {-(4/3)*sqrt(3)})--(6, {(2/3)*sqrt(3)}) node[pos=0.4, above=3.0mm] {$s_2$};
\draw[-latex, color = green, line width=0.8mm] (6, {(2/3)*sqrt(3)})--(7, {-(1/3)*sqrt(3)});
\draw[color = green, line width=0.8mm] (6, {(2/3)*sqrt(3)})--(8, {-(4/3)*sqrt(3)});

\draw[-latex, color = red, line width=0.8mm] (-8, {-(4/3)*sqrt(3)})--(-7, {-(7/3)*sqrt(3)});
\draw[color = red, line width=0.8mm] (-8, {-(4/3)*sqrt(3)})--(-6, {-(10/3)*sqrt(3)});
\draw[-latex, color = red, line width=0.8mm] (-6, {-(10/3)*sqrt(3)})--(-5, {-(7/3)*sqrt(3)});
\draw[color = red, line width=0.8mm] (-6, {-(10/3)*sqrt(3)})--(-4, {-(4/3)*sqrt(3)});
\draw[-latex, color = plum, line width=0.8mm] (-4, {-(4/3)*sqrt(3)})--(-3, {-(7/3)*sqrt(3)});
\draw[color = plum, line width=0.8mm] (-4, {-(4/3)*sqrt(3)})--(-2, {-(10/3)*sqrt(3)});
\draw[-latex, color = plum, line width=0.8mm] (-2, {-(10/3)*sqrt(3)})--(-1, {-(7/3)*sqrt(3)});
\draw[color = plum, line width=0.8mm] (-2, {-(10/3)*sqrt(3)})--(0, {-(4/3)*sqrt(3)});
\draw[-latex, color = green, line width=0.8mm] (0, {-(4/3)*sqrt(3)})--(1, {-(7/3)*sqrt(3)});
\draw[color = green, line width=0.8mm] (0, {-(4/3)*sqrt(3)})--(2, {-(10/3)*sqrt(3)});
\draw[-latex, color = green, line width=0.8mm] (2, {-(10/3)*sqrt(3)})--(3, {-(7/3)*sqrt(3)});
\draw[color = green, line width=0.8mm] (2, {-(10/3)*sqrt(3)})--(4, {-(4/3)*sqrt(3)});
\draw[-latex, color = red, line width=0.8mm] (4, {-(4/3)*sqrt(3)})--(5, {-(7/3)*sqrt(3)});
\draw[color = red, line width=0.8mm] (4, {-(4/3)*sqrt(3)})--(6, {-(10/3)*sqrt(3)});
\draw[-latex, color = red, line width=0.8mm] (6, {-(10/3)*sqrt(3)})--(7, {-(7/3)*sqrt(3)});
\draw[color = red, line width=0.8mm] (6, {-(10/3)*sqrt(3)})--(8, {-(4/3)*sqrt(3)});

\draw[color = plum, line width=0.8mm] (-7, {-(13/3)*sqrt(3)})--(-6, {-(10/3)*sqrt(3)});
\draw[-latex, color = green, line width=0.8mm] (-6, {-(10/3)*sqrt(3)})--(-5, {-(13/3)*sqrt(3)});
\draw[color = green, line width=0.8mm] (-3, {-(13/3)*sqrt(3)})--(-2, {-(10/3)*sqrt(3)});
\draw[-latex, color = red, line width=0.8mm] (-2, {-(10/3)*sqrt(3)})--(-1, {-(13/3)*sqrt(3)});
\draw[color = red, line width=0.8mm] (1, {-(13/3)*sqrt(3)})--(2, {-(10/3)*sqrt(3)});
\draw[-latex, color = plum, line width=0.8mm] (2, {-(10/3)*sqrt(3)})--(3, {-(13/3)*sqrt(3)});
\draw[color = plum, line width=0.8mm] (5, {-(13/3)*sqrt(3)})--(6, {-(10/3)*sqrt(3)});
\draw[-latex, color = green, line width=0.8mm] (6, {-(10/3)*sqrt(3)})--(7, {-(13/3)*sqrt(3)});

%vertices

\node[ver] at (-6, {-(10/3)*sqrt(3)}) (-6, {-(10/3)*sqrt(3)}) {};
\node[ver] at (-2, {-(10/3)*sqrt(3)}) (-2, {-(10/3)*sqrt(3)}) {};
\node[ver] at (2, {-(10/3)*sqrt(3)}) (2, {-(10/3)*sqrt(3)}) {};
\node[ver] at (6, {-(10/3)*sqrt(3)}) (6, {-(10/3)*sqrt(3)}) {};

%\node[circle, draw] at (0, 0) (0, 0) {};
%\node[circle, draw] at (4, 0) (4, 0) {};
%\node[circle, draw] at (2, {(6/3)*sqrt(3)}) (2, {(6/3)*sqrt(3)}) {};

\node[ver] at (-8, {-(4/3)*sqrt(3)}) (-8, {-(4/3)*sqrt(3)}) {};
\node[ver] at (-4, {-(4/3)*sqrt(3)}) (-4, {-(4/3)*sqrt(3)}) {};
\node[ver] at (0, {-(4/3)*sqrt(3)}) (0, {-(4/3)*sqrt(3)}) {};
\node[ver] at (4, {-(4/3)*sqrt(3)}) (4, {-(4/3)*sqrt(3)}) {};
\node[ver] at (8, {-(4/3)*sqrt(3)}) (8, {-(4/3)*sqrt(3)}) {};

\node[ver] at (-6, {(2/3)*sqrt(3)}) (-6, {(2/3)*sqrt(3)}) {};
\node[ver] at (-2, {(2/3)*sqrt(3)}) (-2, {(2/3)*sqrt(3)}) {};
\node[ver] at (2, {(2/3)*sqrt(3)}) (2, {(2/3)*sqrt(3)}) {}; %C0
\node[] at ({2+0.3}, {(2/3)*sqrt(3)-0.2}) ({2+0.3}, {(2/3)*sqrt(3)-0.2}) {$o$}; %C0
\node[ver] at (6, {(2/3)*sqrt(3)}) (6, {(2/3)*sqrt(3)}) {};

\node[ver] at (-8, {(8/3)*sqrt(3)}) (-8, {(8/3)*sqrt(3)}) {};
\node[ver] at (-4, {(8/3)*sqrt(3)}) (-4, {(8/3)*sqrt(3)}) {};
\node[ver] at (0, {(8/3)*sqrt(3)}) (0, {(8/3)*sqrt(3)}) {};
\node[ver] at (4, {(8/3)*sqrt(3)}) (4, {(8/3)*sqrt(3)}) {};
\node[ver] at (8, {(8/3)*sqrt(3)}) (8, {(8/3)*sqrt(3)}) {};

\end{tikzpicture}
}
\caption{A part of the Cayley diagram of the ``triangle" group $\Gamma_{\rm tri}$ with respect to $\{s_1, s_2, s_3\}$ in Example \ref{ex:tri}.
For $i=1, 2, 3$, each $s_i$ acts as a rotation of degree $2\pi/3$ around one of three corners of the triangle indicated by the dark gray.}
\label{fig:tri}
\end{figure}
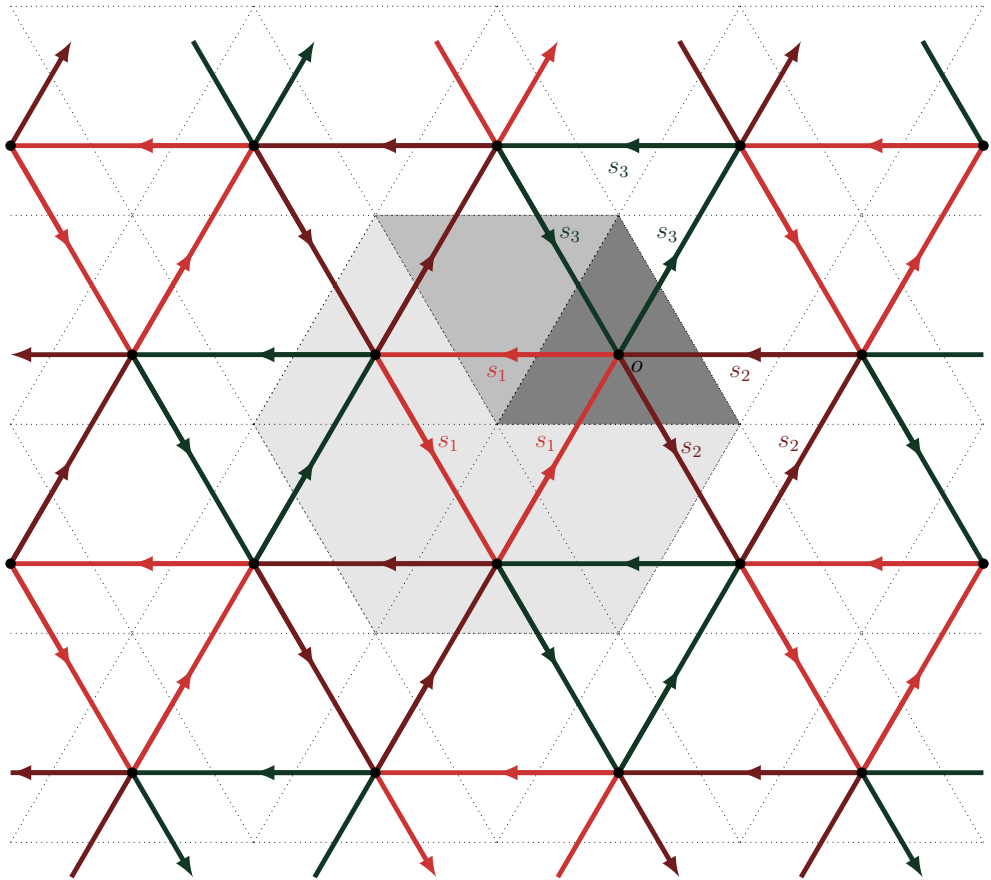

%%end%%tikz%%%%%%%%%%%%%%%%%%%%%%%%%%%%%%%%%%%%%%%%%%%

\begin{example}\label{ex:tri}
Let us consider the following ``triangle'' group
\[
\Gamma_{\rm tri}:=\br{s_1, s_2, s_3 \mid s_1^3=s_2^3=s_3^3=s_1s_2s_3=\id}.
\]
This is an index $2$ subgroup of the affine Weyl group of type $\wt A_2$,
\[
W:=\br{\sigma_1, \sigma_2, \sigma_3 \mid \sigma_1^2=\sigma_2^2=\sigma_3^2=(\sigma_1\sigma_2)^3=(\sigma_2\sigma_3)^3=(\sigma_3\sigma_1)^3=\id}.
\]
The group $\Gamma_{\rm tri}$ is isomorphic to the subgroup generated by $\sigma_1 \sigma_2$, $\sigma_2 \sigma_3$, $\sigma_3 \sigma_1$ in $W$ via the homomorphism induced by $s_i\mapsto \sigma_i \sigma_{i+1}$ where the index $i$ is modulo $3$ for $i=1, 2, 3$.
The affine Weyl group $W$ acts on the Euclidean plane by isometries where $\sigma_i$ acts as a reflection.
The action of $W$ has an equilateral triangle as a fundamental domain (indicated by the dark gray triangle in Figure \ref{fig:tri}),
whereas the action of $\Gamma_{\rm tri}$ has a rhombus as a fundamental domain (indicated by the union of gray and dark gray triangles in Figure \ref{fig:tri}).
The Cayley diagram of $\Gamma_{\rm tri}$ with respect to $\{s_1, s_2, s_3\}$ together with their inverses is indicated in Figure \ref{fig:tri} (where the directed edges corresponding to the inverses $s_i^{-1}$ are omitted).
This is a directed graph with edges labeled by $s_i$ or the inverse (cf.\ Section \ref{sec:cayleydiagram}).
The group $\Gamma_{\rm tri}$ is isomorphic to $\Z^2\rtimes (\Z/3)$, and $\Gamma_{\rm tri}$ does not admit nonzero homomorphisms onto $\Z$ since the group is generated by three elements $s_1$, $s_2$, and $s_3$ of order three.
Let $\mu_{\rm tri}$ be the uniform distribution on $\{s_1, s_2, s_3, s_1^{-1}, s_2^{-1}, s_3^{-1}\}$.
Then, $\mu_{\rm tri}$ is aperiodic since the $\mu_{\rm tri}$-random walk starting from $\id$ returns $\id$ with positive probability after time $2$ and $3$.
Theorem \ref{thm:intro} implies that the $\mu_{\rm tri}$-random walk on $\Gamma_{\rm tri}$ is noise sensitive.
\end{example}

\subsection*{Outlines of the proofs}
Let us briefly explain the proof of Theorems \ref{thm:intro} on noise sensitivity and that of Theorem \ref{thm:zerohom_intro} on more general decoupling.
We will also outline the proof of Theorem \ref{thm:nonzerohom_intro} on non-noise sensitivity.

An infinite finitely generated virtually abelian group $\Gamma$ admits a finite index normal subgroup $\Lambda$ isomorphic to $\mathbb{Z}^m$ for some $m\ge 1$.
The main technical ingredient is to establish the following form of local central limit theorem (CLT) for the $\mu$-random walk on $\Gamma$.
Defining an appropriate ``Gaussian distribution'' $\mathcal{N}_{n, \Sigma, \zeta}$ on $\Gamma$ with mean $n\zeta$ and covariance $n\Sigma$,
we have
\[
\|\mu_n-\mathcal{N}_{n,\Sigma, \zeta}\|_\TV \to 0 \quad \text{as $n \to \infty$}.
\]
In the above, we use the $\Lambda$-coset decomposition to define $\Nc_{n, \Sigma, \zeta}$, as this measure is  a product of an appropriate distribution on the quotient $\Lambda \backslash \Gamma$ and a Gaussian on $\R^m$.

This local CLT holds for an aperiodic $\mu$ with finite second moment, and can be generalized to a non-aperiodic measure once the period is taken into account. See Theorem~\ref{thm:lclt} in Section~\ref{sec:lclt} for the precise statement.
We note that this is a special case of local CLT which has already been proven in \cite{ks} (where the Markov chain is called ``random walks with internal degrees of freedom'').
We, however, establish an explicit formula to compute the covariance matrix $\Sigma$ in terms of harmonic $1$-forms on a finite directed graph associated with the virtually abelian group structure.
This explicit formula is essential to verifying decoupling in Theorem~\ref{thm:zerohom_intro} and noise sensitivity in Theorem~\ref{thm:intro}.
The overall strategy follows the line of proof in the case of affine Weyl groups in \cite{t-ns}; 
however, we need to generalize and strengthen the technical details.

We apply the local CLT to the product group $\Gamma_1 \times \Gamma_2$ and show that if both $\Gamma_1$ and $\Gamma_2$ have no nonzero homomorphisms onto $\Z$, 
then the drift $\zeta$ vanishes and $\Sigma$ has a block diagonal form along the product structure.
(The fact that the drift vanishes follows from the general result \cite{Karlsson-Ledrappier}, see Section \ref{sec:ns}.)
This part requires a group-theoretic consideration.
We consider a canonical group homomorphism $\vartheta: \Gamma \to \Z^m$ called {\it transfer}, defined in terms of the cocycle of the action on $\Lambda\backslash \Gamma$.
If $\Gamma$ admits no nonzero homomorphism onto $\Z$,
then $\vartheta$ has to be the zero homomorphism.
It turns out that this has a strong restriction on the behavior of random walks on $\Gamma$.

If $\Gamma$ has nonzero homomorphisms onto $\Z$,
then the local CLT shows that the covariance matrix detects a correlation between two factors.
This is used to show Theorem \ref{thm:nonzerohom_intro} on non-noise sensitivity.
A part of the proof is reduced to the case of $\Z^m$.

In fact, the explicit form of the covariance matrix of the asymptotic Gaussian distribution allows us to deduce structural results, stated as Theorem~\ref{thm:structure} and Proposition~\ref{thm:structurepi}, related to the transfer homomorphism $\vartheta : \Gamma \to \Z^m$ 
as described in Section~\ref{sec:pre}. 
We will provide the precise statements in Section~\ref{sec:ns}.

\subsection*{Organization} In the preliminary Section~\ref{sec:wdiagram}, we present the notion of weighted diagrams and their harmonic $1$-forms. In Section~\ref{sec:pre}, we devote ourselves to the basics of random walks on virtually abelian groups: the transfer homomorphism, the associated weighted diagram and the potentials corresponding to the cocycle of finite quotient action.
In Section~\ref{sec:lcltmain}, we establish the local CLT for random walks on the groups we are considering, and this will be of independent interest.
In Section \ref{sec:ns},
we establish the general structural result Theorem~\ref{thm:structure}  
on the covariance matrices in the local CLTs. We
apply it to deduce noise sensitivity and decoupling, which provide the proofs of Theorems~\ref{thm:intro}, \ref{thm:nonzerohom_intro} and \ref{thm:zerohom_intro}.

\section{Weighted diagrams and their harmonic $1$-forms}\label{sec:wdiagram}

In this section, we introduce weighted diagrams and their harmonic $1$-forms.

\subsection{Weighted diagram} A diagram $G$ is a multi-digraph with set of vertices $V(G)$ and set of directed edges $E(G)$. To each edge is associated its \emph{original} vertex $oe$ and  its \emph{terminal} vertex $te$.  In our definition, we assume that $E(G)$ is closed under taking reverse edges: to each $e\in E(G)$ there corresponds an edge $\bar{e}$ with $o\bar{e}=te$, $t\bar{e}=oe$ and $\bar{\bar{e}}=e$. Note that we allow the possibility that there are several edges going from a vertex $x=oe$ to another vertex $y=te$, as well as loops, i.e.\ edges with $te=oe$. We will assume that the vertex set $V(G)$ is finite, but even in this case the edge set $E(G)$ can be infinite.

A \emph{weighted diagram} is a diagram equipped with a weight map $p:E(G) \to [0,1]$ such that 
\[
\sum_{e:oe=x} p(e)=1 \quad \text{for all $x \in V(G)$}.
\]
It defines a Markov chain on $V(G)$ with transition probabilities
\[
p(x,y)=\sum_{e:oe=x,te=y} p(e) \quad \text{for all $x,y \in V(G)$}.
\]
If the Markov chain is irreducible, i.e. the probability to reach any vertex starting from any vertex is positive, then by the Perron-Frobenius Theorem, it admits a unique stationary distribution $\pi$ on $V(G)$, i.e.,
\[
\sum_{e : te=x}\pi(oe)p(e)=\pi(x) \quad \text{for each $x\in V(G)$}.
\]
The \emph{conductance} of an edge $e \in E(G)$ is 
$c(e):=\pi(oe)p(e)$. The conductance $c$ can be viewed as a probability measure on $E(G)$. The stationary measure $\pi$ and the weight map~$p$ can be recovered from $c$.
The Markov chain is \emph{reversible} if $c(e)=c(\wb e)$ for every $e \in E(G)$. Note that it is important for our purpose to consider also non-reversible Markov chains on $V(G)$.

\subsection{Harmonic $1$-forms on the weighted diagram $(G, c)$}\label{sec:harmonic}

Let us define the $\C$-linear space of complex-valued functions on $V(G)$,
\[
C^0(G, \C):=\Bigl\{f:V(G) \to \C\Bigr\}.
\]
This space is endowed with the inner product $\br{f_1, f_2}_\pi:=\sum_{x \in V(G)}f_1(x)\wb{f_2(x)}\pi(x)$,
where $\wb \alpha$ denotes the complex-conjugate of a complex number $\alpha$. 
Analogously,
$C^0(G, \R)$ denotes the $\R$-linear space of real-valued functions on $V(G)$ equipped with the inner product~$\br{\cdot,\cdot}_\pi$.
We say that a (real-valued) function on $E(G)$ is a  \emph{$1$-form} if $\omega(\wb e)=-\omega(e)$ for all $e \in E(G)$.
Let $C^1(G, \R)$ denote the space of $1$-forms on $G$.
We define the real Hilbert space of square integrable $1$-forms on $G$ by
\[
\ell_c^2(G, \R):=\Bigl\{\omega \in C^1(G, \R) \ : \ \|\omega\|_c<\infty\Bigr\}.
\]
In the above, $\|\omega\|_c$ denotes the induced norm and the inner product is defined by 
\[
\br{\omega_1, \omega_2}_c:=\frac{1}{2}\sum_{e\in E(G)}\omega_1(e)\omega_2(e)c(e).
\]

For $f\in C^0(G, \R)$,
the \emph{differential} $df$ is the $1$-form defined by
$df(e):=f(te)-f(oe)$ for $e \in E(G)$.
Note that
\begin{align*}
\|df\|_c^2
=\frac{1}{2}\sum_{e \in E(G)}|df(e)|^2 c(e)
&=\frac{1}{2}\sum_{x, y\in V(G)}|f(y)-f(x)|^2\sum_{e: oe=x, te=y}\pi(x)p(e)\\
&\le \frac{1}{2}\sum_{x, y \in V(G)}|f(y)-f(x)|^2\,\pi(y)<\infty.
\end{align*}
In the above, the last summation runs over a finite set.
This shows that 
\[
d:C^0(G, \R) \to \ell_c^2(G, \R) \subset C^1(G,\R)
\]
is a well-defined $\R$-linear operator.
For $\omega \in \ell_c^2(G, \R)$,
let
\[
d^\ast \omega(x):=-\frac{1}{\pi(x)}\sum_{e:oe=x}c(e)\omega(e) \quad \text{for $x \in V(G)$}.
\]
This is well-defined by the Cauchy-Schwarz inequality and defines an $\R$-linear operator $d^\ast: \ell_c^2(G, \R) \to C^0(G, \R)$.
The transition operator $P$ on $C^0(G, \R)$ to itself is defined by
\[
P f(x):=\frac{1}{\pi(x)}\sum_{e: oe=x}c(e)f(te) \quad \text{for $x \in V(G)$}.
\]
Letting $I$ denote the identity operator,
we have
\[
d^\ast d=I-P.
\]

\begin{remark}
When the weighted diagram is reversible, and in this case only, 
the transition operator $P$ is self-adjoint
and $d^\ast$ is the adjoint of $d$, i.e.\ $\br{df, \omega}_c=\br{f, d^\ast \omega}_\pi$ for all $f \in C^0(G, \R)$ and $\omega \in \ell_c^2(G, \R)$.  We insist on the fact that in the general case, $d^\ast$ is \emph{not} the adjoint of $d$.
\end{remark}

For every $\omega \in \ell_c^2(G, \R)$,
let us define
\[
\dr{c}{\omega}:=\sum_{e\in E(G)}\omega(e)c(e).
\]
This is well-defined by the Cauchy-Schwarz inequality.
Let us define the $\R$-linear subspace of \emph{harmonic $1$-forms} by
\[
H^1:=\Bigl\{\omega \in \ell_c^2(G, \R) \ : \ d^\ast \omega+\dr{c}{\omega}=0\Bigr\}.
\]
We have $\dr{c}{\omega}=0$ for all $\omega \in \ell_c^2(G, \R)$ if and only if $c(e)=c(\wb e)$ for all $e \in E(G)$.

\begin{remark}\label{rem:KS} This notion of harmonic $1$-forms was introduced by Kotani and Sunada~\cite{ks-ldp}. In their paper, they highlight the role of the homology class of $\sum_{e \in E(G)}c(e)e$, which they call the {\em homological direction} of a random walk.
\end{remark}
The following lemma appears in \cite[Lemma 5.2]{ks-ldp}, which we prove in an extended form.

\begin{lemma}\label{lem:decomposition}
Letting $\im d:=dC^0(G, \R)$ be the image of $d$ in $\ell_c^2(G, \R)$,
we have the direct sum (but not necessarily orthogonal) decomposition
\[
\ell_c^2(G, \R)=H^1\oplus \im d.
\]
In particular, for every $\omega \in \ell_c^2(G, \R)$, there exists a unique $u \in H^1$ having some $f \in C^0(G, \R)$ such that $\omega=u+df$.
Moreover, $\dr{c}{\omega}=\dr{c}{u}$.
\end{lemma}

\proof
Note that $\im d$ is a finite dimensional closed subspace in $\ell_c^2(G, \R)$.
Furthermore, $H^1$ is a closed subspace in $\ell_c^2(G, \R)$ since $d^\ast \omega_n(x) \to 0$ for each $x \in V(G)$ and $\dr{c}{\omega_n}\to 0$ if a sequence $\omega_n$ tends to $0$ in $\ell_c^2(G, \R)$ by the Cauchy-Schwarz inequality.

Let $\omega \in H^1 \cap \im d$.
Since $\omega \in \im d$,
there exists $f \in C^0(G, \R)$ such that $\omega=df$.
We have $\dr{c}{df}=0$.
Indeed,
\begin{align}\label{eq:lem:decomposition:df}
\dr{c}{df}=\sum_{e\in E(G)}\left(f(te)-f(oe)\right)c(e)
&=\sum_{x\in V(G)}\sum_{e: te=x}f(te)c(e)-\sum_{x \in V(G)}f(x)\pi(x)\nonumber\\
&=\sum_{x\in V(G)}f(x)\left(\sum_{e: te=x}\pi(oe)p(e)-\pi(x)\right)=0.
\end{align}
In the above, the last equality holds since $\pi$ is the stationary distribution.
Therefore, since $\omega=df \in H^1$,
we have
\[
0=d^\ast d f+\dr{c}{df}=(I-P)f.
\]
Noting that $P$ defines an irreducible Markov chain on $V(G)$ shows that $f$ is constant,
whence $\omega=df=0$.
Hence $H^1\cap \im d=\{0\}$.

Let us show that $\ell_c^2(G, \R)=H^1+\im d$.
We will show that for every $\omega \in \ell_c^2(G, \R)$
there exists some $f \in C^0(G, \R)$ such that
\begin{equation}\label{eq:lem:decomposition}
d^\ast \omega+\dr{c}{\omega}-(I-P)f=0.
\end{equation}
This implies the claim.
Indeed, since $\dr{c}{df}=0$ as shown before,
$(I-P)f=d^\ast df$ and
\[
d^\ast(\omega-df)+\dr{c}{\omega-df}=d^\ast \omega+\dr{c}{\omega}-d^\ast df=0,
\]
we have $\omega-df\in H^1$ and $\omega=(\omega-df)+df\in H^1+\im d$.

Let us show \eqref{eq:lem:decomposition}.
Let $P^\ast$ denote the adjoint of $P$ with respect to the inner product in $C^0(G, \R)$.
We have the orthogonal decomposition
$C^0(G, \R)=\im(I-P)\oplus \ker(I-P^\ast)$ in the finite dimensional Hilbert space.
The associated Markov chain with $P^\ast$ (which is the reversed chain of irreducible $P$) is irreducible on $V(G)$,
and thus $\ker (I-P^\ast)$ consists of constant functions $\R\1$.
Therefore, for all $f \in C^0(G, \R)$,
\[
f\in \im(I-P) \iff \br{f, \1}_\pi=0.
\]
Let $\omega \in \ell_c^2(G, \R)$.
By definition of $d^\ast$, we have
\[
\br{d^\ast \omega+\dr{c}{\omega}, \1}_\pi
=-\sum_{x\in V(G)}\sum_{e:oe=x}c(e)\omega(e)+\dr{c}{\omega}\sum_{x\in V(G)}\pi(x)
=-\dr{c}{\omega}+\dr{c}{\omega}=0.
\]
This implies that $d^\ast \omega+\dr{c}{\omega} \in \im(I-P)$,
and thus \eqref{eq:lem:decomposition} holds.
This completes the first claim.
The second claim directly follows from the first claim.
The last claim holds since $\dr{c}{df}=0$ for all $f \in C^0(G, \R)$ by \eqref{eq:lem:decomposition:df}.
\qed

\begin{remark}\label{rem:continuity} In the proof of Lemma \ref{lem:decomposition}, the operator $I-P$ is invertible on the orthogonal complement of $\R\1$ in $C^0(G, \R)$. In the case of Corollary~\ref{cor:rho1}, the conductance $c$ depends continuously on a real one-parameter $\rho \in (0, 1]$, and the stationary distribution on $V(G)$ is uniform independently of $\rho$. If $\omega$ belongs to $\ell_c^2(G, \R)$ for all $c$ in the whole parameter range, the harmonic part $u$ of $\omega$ depends continuously in $\rho \in (0, 1]$. 
Indeed, (cf.\ \eqref{eq:lem:decomposition})
\[
f=(I-P)^{-1}(d^\ast \omega+\dr{c}{\omega})=\sum_{n=0}^\infty P^n(d^\ast \omega +\dr{c}{\omega})
\]depends continuously on $\rho \in (0, 1]$, so does $u=\omega-df$.\end{remark}

\section{Random walks on virtually abelian groups}\label{sec:pre}

An infinite  finitely generated group $\Gamma$ is \emph{virtually abelian} if it admits an abelian subgroup $\Lambda$ of finite index.
Such a  subgroup $\Lambda$ is finitely generated and abelian.
Thus, passing to a subgroup if necessary,
$\Lambda$ is a free abelian group $\Z^m$ for some integer $m\ge 1$
since $\Gamma$ is assumed to be infinite.
Passing to a subgroup further if necessary,
we may assume that $\Lambda$ is a normal subgroup of $\Gamma$.
Namely, one obtains a short exact sequence of groups:
\begin{equation}\label{eq:exactseq}
1 \to \Lambda \to \Gamma \to F \to 1,
\end{equation}
where $\Lambda$ is isomorphic to $\Z^m$ and $F$ is a finite group.
Note that we {\em do not} assume that $\Gamma$ has the form of semidirect product $\Z^m\rtimes F$, except in Lemma~\ref{lem:ucst} and Proposition~\ref{thm:structurepi}.
%Moreover, for the moment (until we discuss random walks), we consider $\Lambda$ merely as a subgroup of $\Gamma$.

Let $\Gamma=\bigsqcup_{x\in \Delta}\Lambda x$ be a left coset decomposition,
where $\Delta$ is a finite set of representatives containing $\id$.
The group $\Gamma$ acting on itself by right translation induces an action on $\Delta$ by permutations $x\mapsto x^\gamma$ for $x \in \Delta$ and $\gamma \in \Gamma$.
Further, this action induces a \emph{cocycle} 
\[
\alpha:\Delta \times \Gamma \to \Lambda, \quad (x, \gamma)\mapsto \alpha(x, \gamma),
\]
where $x\gamma=\alpha(x, \gamma)x^\gamma$ in the unique coset decomposition.
The right action of $\Gamma$ on $\Delta$ reads $x(\gamma \gamma')=(x\gamma)\gamma'$,
from which the \emph{cocycle identity} follows:
\begin{equation}\label{eq:cocycleid}
\alpha(x, \gamma\gamma')=\alpha(x, \gamma)+\alpha(x^\gamma, \gamma')
\quad
\text{for $x \in \Delta$ and $\gamma, \gamma' \in \Gamma$}.
\end{equation}
The additive notation is used for group operation in abelian groups.

\subsection{Transfer and nonzero homomorphisms onto $\Z$}\label{sec:transfer}
The \emph{transfer homomorphism} %(or Verlagerung)
$\vartheta : \Gamma \to \Lambda$ is defined by
\[
\vartheta(\gamma):=\sum_{x\in \Delta}\alpha(x, \gamma).
\]

\begin{lemma}
The transfer $\vartheta : \Gamma \to \Lambda$ defines a group homomorphism. It is independent of the set of representatives $\Delta$.
\end{lemma}

\proof
By the cocycle identity \eqref{eq:cocycleid}, we have
\[
\vartheta(\gamma\gamma')=\sum_{x\in \Delta} \alpha(x,\gamma\gamma')=\sum_{x\in \Delta} \alpha(x,\gamma)+\sum_{x\in \Delta}\alpha(x^\gamma,\gamma')=\vartheta(\gamma)+\vartheta(\gamma')
\]
as $x \mapsto x^\gamma$ is a permutation of $\Delta$.
Let us show that $\vartheta$ is independent of the choice of $\Delta$.
For another set of representatives $\Delta'$,
a natural bijection $\Delta\to \Delta'$, $x\mapsto x'$, is $\Gamma$-equivariant since both $\Delta$ and $\Delta'$ are sets of representatives on the right cosets $\Lambda\backslash \Gamma$.
For each $x$ in $\Delta$,
there exists a unique $\Phi(x')$ in $\Lambda$ such that $x'=\Phi(x')x$.
We have  $x'\gamma=\Phi(x')x\gamma=\Phi(x')\alpha(x, \gamma)x^\gamma$
and, by the $\Gamma$-equivariance of $x\mapsto x'$,
\[
x'\gamma=\alpha(x', \gamma)(x')^\gamma=\alpha(x', \gamma)(x^\gamma)'=\alpha(x', \gamma)\Phi\left((x^\gamma)'\right)x^\gamma.
\]
Hence $\alpha(x, \gamma)=\Phi(x')^{-1}\alpha(x', \gamma)\Phi\left((x^\gamma)'\right)$.
Switching to additive notations and taking sums,
we obtain
\[
\sum_{x\in \Delta}\alpha(x, \gamma)=\sum_{x'\in \Delta'}-\Phi(x')+\alpha(x', \gamma)+\Phi\left((x^\gamma)'\right)=\sum_{x'\in \Delta'}\alpha(x', \gamma),
\]
as required.
\qed

\medskip

In general, the transfer is defined for a group $\Gamma$ and a finite index subgroup $\Lambda$ as a group homomorphism from $\Gamma$ to the abelianization $\Lambda/[\Lambda, \Lambda]$ 
(cf.\ \cite[pp.\ 296-297]{neukirch}, where the definition uses the left cosets instead).
However, we do not discuss the transfer in this generality.

The behavior of the transfer will play a crucial role in our discussion.
On the one hand, if the image $\im(\vartheta)\neq\{0\}$,
then $\Gamma$ admits a nonzero homomorphism onto $\Z$.
On the other hand, if $\im(\vartheta)=\{0\}$,
then for every $\gamma \in \Gamma$,
\[
\sum_{x\in \Delta}\alpha(x, \gamma)=0.
\]
This observation will be used in the subsequent sections.

\subsection{The weighted diagram of a random walk on a virtually abelian group}\label{sec:cayleydiagram}

Let $\mu$ be a probability measure on $\Gamma$. Denote $S=\supp\mu$ (not necessarily symmetric). We assume that $\mu$ is \emph{non-degenerate}, i.e.\ that the sub-semigroup generated by $S$ is $\Gamma$ itself. 
We denote by $\mu_n$ the $n$th convolution power of $\mu$ and $w_n$ the random variable of law $\mu_n$, i.e. the position of the random walk at time $n$.

Let $\Cay(\Gamma, S)$ be the right \emph{Cayley diagram} for $(\Gamma, S\cup S^{-1})$,
i.e.\ the labeled multi-digraph on the set of vertices $\Gamma$ and the set of labeled directed edges from $\gamma$ to $\gamma s$ with label $s$ for each $(\gamma, s) \in \Gamma \times (S\cup S^{-1})$.
The notation $(\gamma, s)$ is used for the directed edge with label~$s$ originating in $\gamma$.

By left translations, every  subgroup $\Lambda$ of $\Gamma$ acts on $\Cay(\Gamma, S)$, preserving the labeled multi-digraph structure. We write $[\gamma]:=\Lambda \gamma$ for the vertex associated to a group element $\gamma$ under the quotient map $\Gamma \to \Lambda\backslash \Gamma$.
Let us define the \emph{quotient diagram} $G:=\Lambda\backslash \Cay(\Gamma, S)$.
This is a labeled multi-digraph on the set of vertices $V(G)=\Lambda\backslash \Gamma$. A directed edge of $G$ is defined by its original vertex $[\gamma]$ and a label $s \in S\cup S^{-1}$. We denote $e=([\gamma], s) \in E(G)$, where $oe:=[\gamma]$ and $te:=[\gamma s]$. The set of labeled directed edges is denoted by
\[
E(G)=\left\{([\gamma], s) \ : \ [\gamma]\in \Lambda\backslash \Gamma, s \in S\cup S^{-1} \right\}.
\]
Furthermore, the reversed edge of $e$ has the inverse label $\wb e=([\gamma s], s^{-1})$.
The set $V(G)$ is finite as we assume $\Lambda$ has finite index in $\Gamma$. The set $E(G)$ is infinite when the measure $\mu$ has infinite support. 

The probability measure $\mu$ on $\Gamma$ induces a weight on $G$ given by
\[
p(e):=\mu(s) \quad \text{for all $e=([\gamma],s) \in E(G)$}.
\]
As $\mu$ is non-degenerate, the induced Markov chain is irreducible. When $\Lambda$ is a normal subgroup (which we will assume) the Markov chain induced on $G$ is identified with the random walk induced in the quotient group $\Lambda \backslash \Gamma=F$. In this case the stationary measure $\pi$ is the uniform measure on $V(G)$.

\subsection{Equivariant maps and potentials}\label{sec:potential}

Let us fix a set $\Delta \subset \Gamma$ of representatives of $\Lambda\backslash \Gamma$. We identify $\Delta$ with $\Lambda\backslash \Gamma=V(G)$ via the quotient map $x \mapsto [x]$. Denote $x_0:=[\id]$. Let us also fix an identification $\Lambda=\Z^m$. The following discussion {\em does} depend on the choice of $\Delta$ and this identification.

Let us define
\[
\Phi: \Gamma \to \Z^m, \quad \gamma=vx \mapsto v,
\]
where $\gamma=vx$ is the unique coset decomposition with $v\in \Lambda=\Z^m$ and $x\in \Delta$.
Note that $\Phi$ is \emph{$\Lambda$-equivariant},
i.e.\ $\Phi(v\gamma)=v+\Phi(\gamma)$ for all $(v, \gamma)\in \Lambda\times \Gamma$.

So for each edge $e=([\gamma],s) \in E(G)$, we may define \[
\Phi_e:=\Phi(t\wt e)-\Phi(o\wt e)=\Phi(\gamma s)-\Phi(\gamma)
\]
where $\tilde{e}=(\gamma,s)$ is an arbitrary lift in the Cayley diagram $\Cay(\Gamma, S)$. Observe that this is the cocycle of the $\Gamma$-action on $\Lambda\backslash\Gamma$:
\begin{equation}\label{eq:phialpha}
\Phi_{([\gamma],s)}=\alpha([\gamma],s).
\end{equation}
Indeed, let $\gamma=v x$ with $v \in \Lambda=\Z^m$ and $x=[\gamma]\in\Delta$. Then $\Phi(\gamma)=v$. For $s \in \Gamma$, we have $\gamma s=v x s =v \alpha(x,s) x^ s$, whence $\Phi(\gamma s)=v+\alpha(x,s)$.

Let us define the vector associated with $\Phi$ by
\[
\zeta_c:=\sum_{e \in E(G)}c(e)\Phi_e.
\]
As will be clear from Theorem~\ref{thm:lclt}, this is the drift of the process $\Phi(w_n)$ taking values in~$\Z^m$.
For each $v \in \R^m$,
let
\[
\wh v(e):=\br{v, \Phi_e} \quad \text{for each $e \in E(G)$},
\]
where $\br{\cdot,\cdot}$ denotes the standard inner product in $\R^m$.
This $\wh v$ defines a $1$-form on $G$ as $\Phi_{\bar{e}}=-\Phi_e$.

Let us fix a word metric $|\cdot|$ in $\Gamma$ and consider a probability measure $\mu$ on $\Gamma$ with finite second moment (see the definition before the statement of Theorem \ref{thm:intro}).

\begin{lemma}\label{lem:second}
If a probability measure $\mu$ on $\Gamma$ has finite second moment,
then the $1$-form $\wh v$ on the quotient diagram $G$ belongs to $\ell_c^2(G, \R)$ for each $v \in \R^m$.
Moreover, $\zeta_c=\sum_{e \in E(G)}c(e)\Phi_e$ is absolutely convergent, and 
$\dr{c}{\wh v}=\br{v, \zeta_c}$	for all $v \in \R^m$.
\end{lemma}

\proof
Let $S_0$ be a finite symmetric set of generators defining the word norm.
Note that $\Phi: (\Gamma, |\cdot|)\to (\R^m, \|\cdot\|)$ is Lipschitz, i.e.\ $\|\Phi(\gamma)\|\le \kappa|\gamma|$ where $\kappa:=\max_{x\in \Delta, s \in S_0}\|\alpha(x, s)\|$.
This follows from the cocycle identity of $\alpha(x, \gamma)$ for $x \in \Delta$ and $\gamma \in \Gamma$ (cf.\ \eqref{eq:cocycleid} and \eqref{eq:phialpha}).
Therefore, we have
\begin{align*}
\sum_{e \in E(G)}|\wh v(e)|^2\,c(e)
&\le \sum_{e \in E(G)}\|v\|^2\|\Phi_e\|^2\,c(e)
=\sum_{x \in V(G)}\sum_{s \in S}\|v\|^2 \|\Phi(s)\|^2\,\pi(x)\mu(s)\\
&\le \|v\|^2\sum_{x \in V(G)}\pi(x)\sum_{s\in S}\kappa^2|s|^2\,\mu(s)=\kappa^2\|v\|^2\sum_{s \in S}|s|^2\,\mu(s)<\infty.
\end{align*}
This shows the first claim.
The second claim follows since
\[
\sum_{e \in E(G)}\|\Phi_e\|\,c(e) \le \kappa\sum_{s \in S}|s|\,\mu(s)<\infty,
\]
and $\dr{c}{\wh{v}}=\sum_{e \in E(G)}c(e)\langle v, \Phi_e\rangle=\langle v, \zeta_c\rangle$.
\qed

\bigskip

We will assume that $\mu$ has finite second moment in the sequel.
The following lemma, which will be used in Section \ref{sec:ns}, indicates how the transfer $\vartheta$ comes into the play (cf.\ Section \ref{sec:transfer}).
Informally speaking, if the transfer $\vartheta$ has zero image, then there is no ``drift'';
patterns of directions encoded by the weighted diagram are canceled out on average.
We show the lemma in a form including the case when $\im(\vartheta)\neq\{0\}$.
Note that the stationary distribution $\pi$ is uniform in the case when $\Lambda\backslash \Gamma=F$ is a finite group (which we assume eventually, cf.\ Section \ref{sec:transfer}).
Recall that $\Lambda$ is identified with $\Z^m$ in $\R^m$, and that $S=\supp \mu$.

\begin{lemma}\label{lem:diagonal}
Let $(G, c)$ be the weighted diagram of the $\mu$-random walk on $\Gamma$. 
Let us consider the orthogonal decomposition
\[
\R^m=\mathrm{span}(\im(\vartheta))\oplus \im(\vartheta)^\perp.
\]
Let $v \in \R^m$.
Denote $\wh{v}=u+df$ where $u$ is the unique harmonic part in $H^1$. Then,
\[
v \in \mathrm{im}(\vartheta)^\perp \quad \Longleftrightarrow  \sum_{x\in V(G)}u(x, s)=0 \quad \text{for all $s \in S$}.
\]
In particular, if $\pi$ is uniform on $V(G)$, then $\dr{c}{\wh{v}}=\dr{c}{u}=0$ for all $v \in \im(\vartheta)^\perp$.
\end{lemma}

\proof
Observe that $\sum_{x\in V(G)}df(x, s)=0$ for each $s\in S$.
Indeed, in $G$, the directed edges labeled $s$ form a (disjoint) union of cycles: $e_0, \dots, e_{n-1}$ such that $o e_i=t e_{i+1}$ for all $i \mod n$.
This follows since the vertex set consists of the right cosets $\Lambda \backslash \Gamma$.
The sum of $df$ along each circle equals $0$,
and thus the indicated sum equals $0$. We have for all $s\in S$,
\[
\sum_{x\in V(G)}u(x,s)=\sum_{x\in V(G)}\wh{v}(x,s)=\sum_{x\in V(G)}\langle v,\alpha(x,s)\rangle=\langle v,\vartheta(s)\rangle.
\]
As $S$ is a generating set of $\Gamma$, the span of $\vartheta(S)$ equals the span of the image of $\vartheta$. This concludes the first claim.

We have $\dr{c}{\wh{v}}=\dr{c}{u}$ by Lemma \ref{lem:decomposition}.
If $\pi$ is uniform on $V(G)$,
then for $v \in \im(\vartheta)^\perp$,
\[
\dr{c}{u}=\sum_{e \in E(G)}c(e)u(e)=\sum_{s\in S}\pi(x)\mu(s)\sum_{x \in V(G)}u(x, s)=0,
\]
where we have used the first claim in the last equality.
\qed

\subsection{The normalized transfer homomorphism as a projection}

The results in this subsection will be used in Section \ref{sec:ns}.

The short exact sequence (\ref{eq:exactseq}) provides an action of $F$ on $\Lambda$ by conjugacy. 
Indeed, if $\gamma=vx$ with $v\in \Lambda$, then for every $\lambda \in \Lambda$, as $\Lambda$ is normal and abelian, we have
\[
\gamma\lambda\gamma^{-1}=vx\lambda x^{-1}v^{-1}=x\lambda x^{-1} \in \Lambda.
\]
This action of $F$ is by automorphisms of the torsion free abelian $\Lambda$, isomorphic to $\mathbb{Z}^m$ for some $m\ge1$. Under this identification, $F$ is isomorphic to a finite subgroup of ${\rm GL}(m,\mathbb{Z})$, acting on $\mathbb{R}^m$ where $\Z^m$ is identified with $\Lambda$. It follows that we can endow $\mathbb{R}^m$ with a scalar product preserved by the $F$-action. We get the adjoint representation $\mathrm{Ad}:F\to {\rm O}(m,\mathbb{R})$ in the orthogonal group.

We define the \emph{normalized transfer homomorphism} $\uth:\R^m \to \R^m$ by
\[
\uth(v):=\frac{1}{\#F}\sum_{f \in F} \mathrm{Ad}(f)v.
\]

\begin{lemma}\label{lem:projector}
The normalised transfer homomorphism $\uth$ is the orthogonal projection of $\R^m$ onto the span of the image of the transfer homomorphism $\vartheta$. In particular, $\mathrm{im}(\uth)=\mathrm{span}(\mathrm{im}(\vartheta))$.
\end{lemma}

\begin{proof}
The scalar product is chosen so that $\mathrm{Ad}(F)$ is orthogonal, so $\uth$ is self-adjoint:
\[
\left\langle \uth(v),u \right\rangle = \frac{1}{\#F}\sum_{f\in F} \langle \mathrm{Ad}(f)v,u\rangle=\frac{1}{\#F}\sum_{f\in F} \langle v, \mathrm{Ad}(f^{-1})u\rangle=\left\langle v,\uth(u)\right\rangle \quad \textrm{for }v,u \in \mathbb{R}^m.
\]
Moreover, as  $f'\mapsto ff'$ is a bijection of $F$,  we have
\begin{align*}
\uth^2(v) &= \frac{1}{(\#F)^2}\sum_{f,f' \in F} \mathrm{Ad}(f) \mathrm{Ad}(f') v=\frac{1}{\#F}\sum_{f \in F} \frac{1}{\#F} \sum_{f'\in F} \mathrm{Ad}(ff') v = \uth(v),
\end{align*}
so $\uth$ is a projection. Finally, as 
\begin{equation}\label{eq:transad}
\alpha(x,v)=xvx^{-1}=\mathrm{Ad}(\Lambda x)v \quad \textrm{for }x\in \Delta, v\in \Lambda,
\end{equation}
we get $\vartheta(v)=\#F\uth(v)$.
\end{proof}

Recall that for $v\in \mathbb{R}^m$, the $1$ form $\wh v(e):=\br{v, \Phi_e}=\br{v,\alpha(x,s)}$, with $e=(x,s)$,
decomposes as $\wh v([x],s)=u(x,s)+df(x,s)$ by Lemma~\ref{lem:decomposition}.

In the next lemma, we assume that the virtually abelian group $\Gamma$ is a semi-direct product $\Gamma =\Lambda \rtimes F$.

\begin{lemma}\label{lem:ucst}
Assume that $\Gamma =\Lambda \rtimes F$ is a semi-direct product and that $\Delta$ is a subgroup isomorphic to $F$. Let $v\in \mathrm{im}(\uth)$. Then for each $s \in \Gamma$, the harmonic part $u(x,s)$ of the $1$-form $\wh v$ does not depend on $x$.
\end{lemma}

\begin{proof}
Write $s=\lambda y$ with $\lambda \in \Lambda$ and $y \in \Delta$. By the cocycle relation we have
\[
\wh v(x,s)=\langle v, \alpha(x,\lambda y)\rangle = \langle v,\alpha(x,\lambda)\rangle +\langle v,\alpha(x,y)\rangle,
\]
as $x^\lambda=x$ and $x\lambda=(x\lambda x^{-1})x$, where we note that $\Lambda$ is a normal subgroup in $\Gamma$. 
On the one hand, using Lemma~\ref{lem:projector} and (\ref{eq:transad}), we have
\begin{align*}
\langle v,\alpha(x,\lambda)\rangle &=\langle \uth(v),\mathrm{Ad}(\Lambda x) \lambda \rangle = \langle v,\uth\mathrm{Ad}(\Lambda x) \lambda \rangle \\ &= \left\langle v,\frac{1}{\#F} \sum_{f \in F}\mathrm{Ad}(f)\mathrm{Ad}(\Lambda x)\lambda \right\rangle=\langle v,\uth(\lambda)\rangle,
\end{align*}
independent of $x$.
On the other hand, as $\Gamma$ is a semi-direct product with $\Delta$ isomorphic to $F$, we have $\alpha(x,y)=\Phi(xy)-\Phi(x)=0$, because $xy\in \Delta$ and $\Phi(x)=0$ for all $x \in \Delta$ by definition.
\end{proof}

\section{Local CLTs on virtually abelian groups}\label{sec:lcltmain}

We establish local CLTs for random walks on finitely generated infinite virtually abelian groups.
In this section, we keep the notations for weighted diagrams $(G, c)$ associated to random walks on virtually abelian groups in the previous Section~\ref{sec:pre}.
In Section \ref{sec:perturbation},
we define transfer operators, discuss their perturbations on weighted diagrams,
and use them to establish the variance formula in terms of harmonic forms in the local CLT.
In Section \ref{sec:char},
we introduce characteristic functions to apply the Fourier analysis and discuss the periodicity problem arising in Markov chains.
In Section \ref{sec:lclt},
we establish the local CLT based on the transfer operator method.

\subsection{Transfer operators and their perturbations}\label{sec:perturbation}

For each $\omega \in C^1(G, \R)$,
the \emph{transfer operator} on $C^0(G, \C)$ is defined by
\[
\Lc_\omega f(x):=\sum_{e:oe=x}p(e)e^{2\pi i\omega(e)}f(te) \quad \text{for $x\in V(G)$}.
\]
In the definition, $i=\sqrt{-1}$ and $\pi$ denotes the half of the circumference of the unit circle.
Note that the operator norm of $\Lc_\omega$ is uniformly bounded $\|\Lc_\omega\|\le 1$ for all $\omega \in C^1(G, \R)$.
The map $\omega\mapsto \Lc_\omega$ defines a twice continuously differentiable ($C^2$-)map from $\ell_c^2(G, \R)$ to the complex Banach space of bounded operators on $C^0(G, \C)$, where the space of operators is a finite dimensional space of matrices.
Since $\Lc_0=P$ and $P$ has a simple eigenvalue $1$,
we have an open neighborhood $U$ of $0$ in $\ell_c^2(G, \R)$ and a $C^2$-function $\lambda: U \to \R$ such that $\Lc_\omega$ has a simple eigenvalue $\lambda(\omega)$ for all $\omega \in U$ and $\lambda(0)$=1.
This follows from the implicit function theorem in the Hilbert space applied to the characteristic polynomial.
Furthermore, the corresponding eigenfunction $f_\omega$ can be chosen so that $\omega\mapsto f_\omega$ is a $C^2$-map from $U$ to $C^0(G, \R)$ and that $f_\omega$ is normalized, i.e.\ $\br{f_\omega, \1}_\pi=1$ for all $\omega \in U$.
In summary,
we have $\lambda(0)=1$, $f_0=\1$,
\[
\Lc_\omega f_\omega=\lambda(\omega)f_\omega \quad \text{and} \quad \br{f_\omega, \1}_\pi=1 \quad \text{for all $\omega \in U$}.
\]
Moreover, $\lambda(\omega)$ and $f_\omega$ are $C^2$-maps in $\omega \in U$ respectively. Replacing $U$ with a smaller open neighborhood of $0$ if necessary, we take a branch of the logarithm and define
\[
\beta(\omega):=\log \lambda(\omega)\quad \text{for $\omega \in U$ such that $\beta(0)=0$}.
\]

The following observation is crucial to our discussion:
If $\omega \in \ell_c^2(G, \R)$ and $\f \in C^0(G, \R)$,
then $\omega+d\f\in \ell_c^2(G, \R)$ and $\Lc_{\omega+d\f}=e^{-2\pi i \f}\Lc_\omega e^{2\pi i\f}$ where $e^\f f(x):=e^{\f(x)}f(x)$ for $x \in V(G)$.
Therefore for all $\omega \in U$ and all $\f \in C^0(G, \R)$,
\[
\lambda(\omega+d\f)=\lambda(\omega).
\]
This implies that $\lambda(\omega)$ and thus $\beta(\omega)$ depend only on the harmonic part of $\omega$ in the decomposition (cf.\ Lemma \ref{lem:decomposition}).

\begin{lemma}\label{lem:pressure}
Let $r, r_i$, $i=1, 2$, be real parameters.
\begin{itemize}
\item[(1)]\label{it:1}
For every $1$-form $\omega \in \ell_c^2(G, \R)$,
\begin{equation}\label{eq:lem:pressure1}
\frac{d}{dr}\Big|_{r=0}\beta(r\omega)=2\pi i\dr{c}{\omega}.
\end{equation}
\item[(2)]\label{it:2}
For every harmonic $1$-form $u$ in $\ell_c^2(G, \R)$,
\begin{equation}\label{eq:lem:eigen}
\frac{d}{dr}\Big|_{r=0}f_{ru}(x)=0 \quad \text{for all $x\in V(G)$}.
\end{equation}
\item[(3)]\label{it:3}
For all harmonic $1$-forms $u_i$, $i=1, 2$, in $\ell_c^2(G, \R)$,
\begin{align}\label{eq:lem:pressure2}
\frac{\partial^2}{\partial r_1\partial r_2}&\Big|_{(r_1, r_2)=(0, 0)}\beta(r_1 u_1+r_2 u_2) \nonumber\\ &=-4\pi^2\left(\sum_{e\in E(G)}u_1(e)u_2(e)c(e)-\sum_{e\in E(G)}u_1(e)c(e)\sum_{e\in E(G)}u_2(e)c(e)\right)
\end{align}
\item[(4)]\label{it:4}
If $c(e)=c(\wb e)$ for all $e \in E(G)$,
then all (existing) odd time derivatives of $\beta$ at $0$ vanish.
\end{itemize}
\end{lemma}

\proof
For each $1$-form $\omega \in U$,
we have $r\omega \in U$ for all small enough $r$.
The normalization $\br{f_{r\omega}, \1}_\pi=1$ implies that 
\begin{equation}\label{eq:lem:eigen1}
\br{\frac{d}{dr}\Big|_{r=0}f_{r\omega}, \1}_\pi=0.
\end{equation}
Furthermore, since $P^\ast \1=\1$,
this \eqref{eq:lem:eigen1} shows that
\begin{equation}\label{eq:lem:eigen2}
\br{P\left(\frac{d}{dr}\Big|_{r=0}f_{r\omega}\right), \1}_\pi
=\br{\frac{d}{dr}\Big|_{r=0}f_{r\omega}, P^\ast \1}_\pi
=\br{\frac{d}{dr}\Big|_{r=0}f_{r\omega}, \1}_\pi=0.
\end{equation}
Note that identities analogous to \eqref{eq:lem:eigen1} and \eqref{eq:lem:eigen2} also  hold  for second derivatives (and in fact for any order derivatives if they exist) by the normalization condition on $f_{r\omega}$.

Differentiating $\Lc_{r\omega}f_{r\omega}=e^{\beta(r\omega)}f_{r\omega}$ at $r=0$, we obtain for each $x \in V(G)$,
\begin{equation}\label{eq:lem:pressure:diff}
\sum_{e:oe=x}\left(p(e)(2\pi i\omega(e))+p(e)\frac{d}{dr}\Big|_{r=0}f_{r\omega}(te)\right)
=\frac{d}{dr}\Big|_{r=0}\beta(r\omega)+\frac{d}{dr}\Big|_{r=0}f_{r\omega}(x),
\end{equation}
where we have used $f_0=\1$ and $\beta(0)=0$.
Taking the inner products with $\1$ in both sides yields by \eqref{eq:lem:eigen1} and \eqref{eq:lem:eigen2},
\[
\sum_{x\in V(G)}\pi(x)\sum_{e:oe=x}p(e)(2\pi i\omega(e))=\frac{d}{dr}\Big|_{r=0}\beta(r\omega).
\]
The left hand side equals $2\pi i \dr{c}{\omega}$, and this shows \eqref{eq:lem:pressure1}.

Let us take a harmonic $1$-form $\omega=u$ in $\ell_c^2(G, \R)$.
By \eqref{eq:lem:pressure:diff} and \eqref{eq:lem:pressure1},
we have
\[
-2\pi id^\ast u (x)+P\left(\frac{d}{dr}\Big|_{r=0}f_{ru}\right)(x)=2\pi i \dr{c}{u}+\frac{d}{dr}\Big|_{r=0}f_{ru}(x).
\]
Since $u$ is a harmonic $1$-form, i.e.\ $d^\ast u+\dr{c}{u}=0$,
the above equation implies that
\[
P\left(\frac{d}{dr}\Big|_{r=0}f_{ru}\right)(x)=\frac{d}{dr}\Big|_{r=0}f_{ru}(x) \quad \text{for each $x \in V(G)$}.
\]
Furthermore, since $P$ is irreducible, $(d/dr)|_{r=0}f_{r u}$ is a constant, which has to be $0$ by \eqref{eq:lem:eigen1}.
This shows \eqref{eq:lem:eigen}.

For all small enough $r_i$ for $i=1, 2$,
letting $\beta_{r_1, r_2}:=\beta(r_1u_1+r_2u_2)$ and $f_{r_1, r_2}:=f_{r_1u_1+r_2u_2}$ for harmonic $1$-forms $u_i$,
we have for each $x\in V(G)$,
\[
\Lc_{r_1 u_1+r_2 u_2}f_{r_1, r_2}(x)=e^{\beta_{r_1, r_2}}f_{r_1, r_2}(x).
\]
Computing the second derivatives at $(r_1, r_2)=(0, 0)$ yields by \eqref{eq:lem:eigen},
\begin{align*}
&-4\pi^2\sum_{e\in E_x}p(e)u_1(e)u_2(e)
+\sum_{e:oe=x}p(e)\frac{\partial^2}{\partial r_1\partial r_2}\Big|_{(0, 0)}f_{r_1, r_2}(x)\\
&\qquad \qquad=\frac{\partial^2}{\partial r_1\partial r_2}\Big|_{(0, 0)}\beta_{r_1, r_2}+\frac{\partial^2}{\partial r_1\partial r_2}\Big|_{(0, 0)}f_{r_1, r_2}(x) + \left(\frac{\partial}{\partial r_1}\Big|_{(0, 0)}\beta_{r_1, r_2} \right)\left(\frac{\partial}{\partial r_2}\Big|_{(0, 0)}\beta_{r_1, r_2}\right).
\end{align*}
Taking the inner products with $\1$ in both sides and using (\ref{eq:lem:pressure1}) yields
\[
-4\pi^2\left(\sum_{e\in E(G)}c(e)u_1(e)u_2(e)-\sum_{e\in E(G)}c(e)u_1(e)\sum_{e\in E(G)}c(e)u_2(e)\right)=\frac{\partial^2}{\partial r_1\partial r_2}\Big|_{(r_1, r_2)=(0, 0)}\beta_{r_1, r_2}.
\]
To deduce the equality, we have used the analogous identities for second derivatives to \eqref{eq:lem:eigen1} and \eqref{eq:lem:eigen2}.
This concludes \eqref{eq:lem:pressure2}.

Furthermore, if $c(e)=c(\wb e)$ for all $e \in E(G)$,
then $\Lc_\omega$ is self-adjoint for each $1$-form~$\omega$.
Hence $\lambda(\omega)$ is real,
and since $\wb{\Lc_\omega f_\omega}=\Lc_{-\omega}\wb{f_\omega}$,
we have $\lambda(-\omega)=\wb{\lambda(\omega)}=\lambda(\omega)$ for all $\omega \in U$.
This implies that all odd time derivatives of $\beta$ at $0$ vanish, concluding the last claim.
\qed

\bigskip

Let us consider the transfer operators $\Lc_{\wh v}$ with \emph{potentials} $\wh v$ associated with $\Phi$ for $v \in \R^m$.
For brevity,
let us denote $\lambda(v):=\lambda(\wh v)$ and $\beta(v):=\beta(\wh v)$
for all $v$ near $0$ in $\R^m$.

\begin{lemma}\label{lem:hessian}
Let us denote the Hessian of $\beta$ at $0$ in $\R^m$ by
\[
\Hess_0\beta:=\left(\frac{\partial^2}{\partial r_i \partial r_j}\Big|_{(r_1, \dots, r_m)=(0, \dots, 0)}\beta(r_1, \dots, r_m)\right)_{i, j=1, \dots, m}.
\]
The $\Hess_0 \beta$ is non-degenerate and negative definite, and satisfies
\begin{equation}\label{eq:lem:hessian}
\br{v_1, \Hess_0 \beta\, v_2}=-4\pi^2 \left(\sum_{e\in E(G)}u_1(e)u_2(e)c(e)-\sum_{e\in E(G)}c(e)u_1(e)\sum_{e\in E(G)}c(e)u_2(e)\right)
\end{equation}
where $u_i$ is the unique harmonic part of $\wh v_i$ for all $v_i \in \R^m$, $i=1, 2$.
\end{lemma}

\proof
For all $v \in \R^m$,
the associated $1$-form $\wh v$ is in $\ell_c^2(G, \R)$ by Lemma \ref{lem:second},
and thus $\wh v$ has a unique harmonic part $u$ by Lemma \ref{lem:decomposition}.
Lemma \ref{lem:pressure} \eqref{eq:lem:pressure2} implies \eqref{eq:lem:hessian}.
Suppose that $\br{v, \Hess_0\beta\,v}=0$.
By \eqref{eq:lem:hessian}, we have
\[
0=\sum_{e\in E(G)} u(e)^2c(e)-\Bigg(\sum_{e\in E(G)}u(e)c(e)\Bigg)^2
=\sum_{e \in E(G)}c(e)\Bigg(u(e)-\sum_{e\in E(G)}c(e)u(e)\Bigg)^2.
\]
The unique harmonic part $u$ of $\wh v$ is constant, hence zero, because it is a $1$-form; and thus $\wh v=df$ for some $f \in C^0(G, \R)$ by Lemma \ref{lem:decomposition}.
Let us consider a directed edge path in $\Cay(\Gamma, S)$ from $\id$ to $w \in \Z^m=\Lambda$ in $\Gamma$ and its projection $(e_1, \dots, e_n)$ in $E(G)$.
Since the projection is a loop,
we have (cf.\ \eqref{eq:period} below for details about the fourth equality)
\[
0=\sum_{i=1}^n df(e_i)=\sum_{i=1}^n \wh v(e_i)=\sum_{i=1}^n\br{v, \Phi_{e_i}}=
\br{v, w}.
\]
Running $w$ over a basis in $\Z^m$ yields $v=0$.
Therefore $\Hess_0 \beta$ is non-degenerate.
The formula \eqref{eq:lem:hessian} shows that $\Hess_0\beta$ is negative definite.
\qed

\subsection{Characteristic functions}\label{sec:char}

Let us define $\1_x:V(G) \to \R$ by $\1_x(y)=1$ for $y=x$ and $0$ for $y\neq x$.
For each positive integer $n$ and for all $x \in V(G)$, we can compute explicitly iterates of the perturbed transfer operator
\[
\Lc_{\wh v}^n\1_{x}(x_0)=\sum p(e_1)\cdots p(e_n)\exp\left(2\pi i (\wh v(e_1)+\cdots+\wh v(e_n))\right).
\]
The above summation runs over all paths $(e_1, \dots, e_n)$ from $x_0$ to $x$, $o e_1=x_0$, $te_i=o e_{i+1}$ for $i=1, \dots, n-1$ and $te_n=x$.
A lift of $(e_1, \dots, e_n)$ is defined by a path $(\wt e_1, \dots, \wt e_n)$ in $\Cay(\Gamma, S)$ with $t\wt e_i=o\wt e_{i+1}$ for $i=1, \dots, n-1$
and $\wt e_i \mapsto e_i$ via the covering map $\Cay(\Gamma, S) \to \Lambda\backslash \Cay(\Gamma, S)$.
For such a lift, we have
\begin{equation}\label{eq:period}
\sum_{i=1}^n \wh v(e_i)=\sum_{i=1}^n \br{v, \Phi_{e_i}}=\sum_{i=1}^n \br{v, \Phi(t\wt e_i)-\Phi(o\wt e_i)}=\br{v, \Phi(t\wt e_n)-\Phi(o\wt e_1)}.
\end{equation}
Taking all the lifts starting from $\id$,
we obtain by $o\wt e_1=\id$ and $\Phi(\id)=0$,
\[
\Lc_{\wh v}^n \1_{x}(x_0)=\sum \mu(s_1)\cdots \mu(s_n)\exp\left(2\pi i\br{v, \Phi(t\wt e_n)}\right)=\sum_{\gamma \in \Gamma, [\gamma]=x}\mu_n(\gamma)\exp\left(2\pi i\br{v, \Phi(\gamma)}\right),
\]
where $s_i$ are the labels of each path read along.
The second equality holds since the map assigning to a path from $x_0$ to $x$ a lift starting from $\id$ of length $n$ terminating at $\gamma$ with $[\gamma]=x$ is a bijection.
Note that $\Lc_0^n\1_{x}(x_0)$ is simply the probability that the random walk on $\Gamma$ started in $\id$ projects to $x \in V(G)$ at time $n$.

We consider the distribution $\Phi_\ast \mu_n:=\mu_n\circ \Phi^{-1}$ of $\Phi(w_n)$, which is the image by $\Phi$ into $\Lambda=\Z^m$ of the $\mu$-random walk on $\Gamma$ at time $n$. Its \emph{characteristic function} is given by
\[
\f_{\Phi_\ast\mu_n}(v):=\sum_{\gamma \in \Gamma} \mu_n(\gamma)\exp\left(2\pi i \langle v,\Phi(\gamma)\rangle \right)=\sum_{x\in V(G)} \Lc_{\wh v}^n \1_{x}(x_0).
\]
Since $\Phi(\Gamma)=\Z^m$,
the function $\f_{\Phi_\ast\mu_n}$ is defined on $\R^m/\Z^m$,
which will also be identified with a fundamental domain
\[
D:=\left[-\frac{1}{2}, \frac{1}{2}\right)^m \quad \text{in $\R^m$}.
\]
The Fourier inversion formula yields
\[
(\Phi_\ast\mu_n)(\lambda)=\int_{\R^m/\Z^m} \sum_{x\in V(G)} \Lc_{\wh v}^n \1_{x}(x_0) e^{-2\pi i\langle v,\lambda\rangle}dv \quad \text{for all $\lambda \in \Z^m$}.
\]
It also gives the following fact.

\begin{fact}
For each $\gamma \in \Gamma$,
\begin{equation}\label{eq:fourier_inversion}
\mu_n(\gamma)=\int_{\R^m/\Z^m} \Lc^n_{\wh{v}}\1_{[\gamma]}(x_0)e^{-2\pi i\langle v, \Phi(\gamma)\rangle}dv.
\end{equation}
\end{fact}

\proof
Let us compute for each $\gamma_0 \in \Gamma$,
\begin{equation*}
\int_{\R^m/\Z^m}  \Lc^n_{\wh{v}}\1_{x}(x_0)e^{-2\pi i\langle v, \Phi(\gamma_0)\rangle}dv 
%& = \int_{\R^m/\Z^m} \sum_{\gamma \in \Gamma, [\gamma]=x}\mu_n(\gamma) e^{2\pi i\langle v,\Phi(\gamma)- \Phi(\gamma_0)\rangle}dv \\&
=\sum_{\gamma \in \Gamma, [\gamma]=x}\mu_n(\gamma)  \int_{\R^m/\Z^m} e^{2\pi i\langle v,\Phi(\gamma)- \Phi(\gamma_0)\rangle}dv.
\end{equation*}
The series is absolutely convergent as $\mu_n$ is a probability measure. The integral term vanishes when $\Phi(\gamma)-\Phi(\gamma_0)\neq 0$ as $\Phi$ takes integer values, and otherwise it takes value one. If $x=[\gamma_0]$, we get $\mu_n(\gamma_0)$ as required. Otherwise all terms are zero.
\qed

\medskip

We define the \emph{period} $q$ of a probability measure $\mu$ on $\Gamma$ by the greatest common divisor of all possible lengths of nontrivial directed loops at $\id$ in $\Cay(\Gamma, S)$ where $S=\supp \mu$.
We say that $\mu$ is \emph{aperiodic} if the period of $\mu$ is $1$.
This is equivalent to say that there exists $N$ such that $\mu_n(\id)>0$ for every $n \ge N$.
For example, $\mu$ is aperiodic if $\mu(\id)>0$.

For $\delta>0$,
let
\[
D_\delta:=\Bigl\{v=(v_1, \dots, v_m) \in \R^m \ : \ |v_i|<\delta, \ i=1, \dots, m\Bigr\}.
\]

\begin{lemma}\label{lem:aperiodic}
If $\mu$ is aperiodic,
then for each $\delta>0$,
there exist constants $c_\delta, C_\delta>0$ such that for all $n\in \N$ and for all $x\in V(G)$,
\[
\sup_{v \in D\setminus D_\delta}|\Lc_{\wh v}^n \1_{x}(x_0)|\le C_\delta \exp\left(-c_\delta n\right).
\]
\end{lemma}

\proof
Note that $v \mapsto \Lc_{\wh v}$ is continuous and that $\|\Lc_{\wh v}\|:=\max_{\|f\|=1}\|\Lc_{\wh v}f\|$ depends continuously on $v \in \R^m$,
where $\|\cdot\|$ denotes the induced norm in $C^0(G, \C)$.
The spectral radius is obtained by $\rad(v)=\lim_{n \to \infty}\|\Lc_{\wh v}^n\|^{1/n}$ for each $v \in \R^m$,
where the limit exists by the sub-multiplicativity of $n \mapsto \|\Lc_{\wh v}^n\|$.
This shows that $v \mapsto \rad(v)$ is upper semi-continuous,
in particular, $\rad(v)$ attains its maximum on each compact set.

If $\mu$ is aperiodic,
then $\rad(v)<1$ for $v \in D\setminus\{0\}$.
Indeed, let $\lambda(v)$ be the maximal eigenvalue in absolute value and $f_v$ be a corresponding (nonzero) eigenfunction.
We have $|\lambda(v)|=\rad(v)$ and $\Lc_{\wh v}f_v=\lambda(v)f_v$ by definition.
If $|\lambda(v)|=1$,
then $|f_v|\le P|f_v|$.
Since $P$ is irreducible on the finite set $V(G)$,
the function $|f_v|$ is a nonzero constant.
Thus, $\lambda(v)f_v(x)$ and $e^{2\pi i\br{v, \Phi_e}}f_v(te)$ lie on a common circle in the complex plane for all $x\in V(G)$ and all $e\in E_x$.
Since $\lambda(v)f_v(x)=\sum_{e:oe=x}p(e)e^{2\pi i\br{v, \Phi_e}}f_v(te)$,
we have
\[
\lambda(v)f_v(x)=e^{2\pi i \br{v, \Phi_e}}f_v(te) \quad \text{for all $x \in V(G)$ and all $e$ with $oe=x$}.
\]
If $\mu$ is aperiodic,
then for each $w \in \Z^m=\Lambda$
there exists a directed path from $\id$ to $w$ in $\Cay(\Gamma, S)$ of all large enough length $n$.
Projecting the path to $G$, we have $\Phi(w)=w$ since we assume that $\id \in \Delta$,
and $\lambda(v)^nf_v(x_0)=e^{2\pi i\br{v, w}}f_v(x_0)$.
This yields $\lambda(v)^n=\lambda(v)^{n+1}$ by taking paths of length $n$ and $n+1$ from $\id$ to $w$, and thus $\lambda(v)=1$ and $\br{v, w} \in \Z$ since $|f_v|$ is a nonzero constant.
Taking $w$ over a basis of $\Z^m$,
we obtain $v=0$ since $v \in D$.
This shows that if $|\lambda(v)|=1$, then $v=0$.

Therefore, if $\mu$ is aperiodic,
then for each $\delta>0$, there exists a constant $c_\delta>0$ such that $\rad(v)\le e^{-c_\delta}$ for all $v \in D\setminus D_\delta$.
Furthermore, for each $v \in D\setminus D_\delta$,
there exists $N=N_v$ such that
$\|\Lc_{\wh v}^{N}\|^{1/N} \le e^{-c_\delta/2}$.
Since $v \mapsto \Lc_{\wh v}$ is continuous,
there exists an open set $U_v$ in $D\setminus D_\delta$ containing $v$ such that $\|\Lc_{\wh w}^{N}\|^{1/N}\le e^{-c_\delta/4}$ for every $w \in U_v$.
For every $n \in \N$, 
letting $n=k N+l$ for some $k \in \N$ and some $0 \le l<N$,
we have
\[
\|\Lc_{\wh w}^n\|\le \|\Lc_{\wh w}\|^l \cdot \|\Lc_{\wh w}^{N}\|^k
\le \exp\left(-\frac{c_\delta}{4}k N\right)= e^{c_\delta l/4}\exp\left(-\frac{c_\delta}{4}n\right) \le e^{c_\delta N/4}\exp\left(-\frac{c_\delta}{4}n\right).
\]
In the second inequality, we have used $\|\Lc_{\wh w}\|\le 1$.
By compactness of $D\setminus D_\delta$ under the identification $D=\R^m/\Z^m$,
we have finitely many open sets $U_v$ that cover $D\setminus D_\delta$.
Hence there exist constants $c_\delta', C_\delta>0$ such that
$\|\Lc_{\wh w}^n\|\le C_\delta \exp\left(-c_\delta' n\right)$ for all $w \in D\setminus D_\delta$ and for all $n \in \N$. 
We have a constant $C_\delta'$ such that
\[
|\Lc_{\wh v}^n \1_{x}(x_0)|\le \|\Lc_{\wh v}^n\|\sqrt{\pi(x)/\pi(x_0)}
\le C_\delta'\exp(-c_\delta' n)
\]
for all $v \in D\setminus D_\delta$, for all $x\in V(G)$ and all $n\in \N$, showing the claim.
\qed

\begin{remark}
Lemma \ref{lem:aperiodic} and its proof are known in the literature e.g.\ \cite[p.93, Section 5]{ks}.
We have provided a self-contained proof adapted to the current setting.
\end{remark}

\subsection{A concentration inequality}

Recall the $\Lambda$-equivariant map $\Phi:\Gamma \to \Z^m$ and $\zeta_c=\sum_{e \in E(G)}c(e)\Phi_e$ (cf.\ Section \ref{sec:potential}). The image under $\Phi$ of the random walk at time $n$ is concentrated around $n\zeta_c$.

\begin{lemma}\label{lem:martingale}
Let $\mu$ be a probability measure on $\Gamma$ of finite second moment,
and $\{w_n\}_{n \in \N}$ be a $\mu$-random walk starting from $\id$ on $\Gamma$.
There exists a constant $C>0$ such that for all $n \in \N$ and all real $r>0$,
\[
\Prob\Bigl(\|\Phi(w_n)-n\zeta_c\| \ge r \Bigr)\le \frac{C n}{r^2}.
\]
\end{lemma}

\proof
For each $v \in \R^m$,
we have $\wh v \in \ell_c^2(G, \R)$ by Lemma \ref{lem:second}.
By Lemma \ref{lem:decomposition},
we have $\wh v=u+df$,
where $u$ is the unique harmonic part of $\wh v$ and $f \in C^0(G, \R)$.
The definition of the harmonic part, the last claim in Lemma \ref{lem:decomposition} and Lemma \ref{lem:second} imply that
\[
0=d^\ast u+\dr{c}{u}=d^\ast u+\dr{c}{\wh v}=d^\ast u+\br{v, \zeta_c}.
\]
Taking lifts $\wt v$, $\wt u$ and $\wt f$ of $\wh v$, $u$ and $f$ respectively,
via $\Cay(\Gamma, S) \to \Lambda\backslash \Cay(\Gamma, S)$,
we obtain
\[
\wt v=\wt u+d\wt f.
\]
Abusing notations, we use $1$-forms and differential $d$ defined on $G$ also on $\Cay(\Gamma, S)$.
For each $\gamma \in \Gamma$, we have
\[
-\sum_{s \in S}\wt u(\gamma, s)\mu(s)+\br{v, \zeta_c}=0.
\]
For a $\mu$-random walk $w_n=\gamma_1\cdots \gamma_n$ starting from $\id$ on $\Gamma$,
for $n \ge 1$,
let
\[
\wt u_n:=\sum_{i=1}^n\Bigl(\wt u(w_{i-1}, \gamma_i)-\br{v, \zeta_c}\Bigr) \quad \text{and} \quad \wt u_0:=0.
\]
This defines a martingale $\{\wt u_n\}_{n \in \N}$ with respect to the natural filtration $\{\Fc_n\}_{n \in \N}$ associated with $\{w_n\}_{n \in \N}$,
i.e., $\E[\wt u_{n+1}\mid \Fc_n]=\wt u_n$ almost surely for each $n \in \N$.
Furthermore, the martingale difference $\wt u_{n+1}-\wt u_n$ is square integrable since $u \in \ell_c^2(G, \R)$ and $\wt u$ is a lift of $u$.
Note that for each $\gamma \in \Gamma$ and for each path $e_1, \dots, e_n$ in $\Cay(\Gamma, S)$ with $oe_1=\id$, $te_i=oe_{i+1}$ for $i=1, \dots, n-1$ and $te_n=\gamma$,
\[
\br{v, \Phi(\gamma)-\Phi(\id)-n\zeta_c}=\sum_{i=1}^n\Bigl(\wt v(e_i)-\br{v, \zeta_c}\Bigr)=\sum_{i=1}^n \Bigl(\wt u(e_i)+d\wt f(e_i)-\br{v, \zeta_c}\Bigr).
\]
The last term equals
$\wt u_n+\wt f(w_n)-\wt f(\id)$ for a random walk path.
This implies that since $\Phi(\id)=0$, for all $n \in \N$,
\begin{equation}\label{eq:lem:martingale}
\br{v, \Phi(w_n)-n\zeta_c}=\wt u_n+\wt f(w_n)-\wt f(\id).
\end{equation}
Moreover, $\wt f$ is uniformly bounded, i.e., $\|\wt f\|_\infty=\|f\|_\infty<\infty$ since $\wt f$ is a lift of $f$.

Since $\{\wt u_n\}_{n \in \N}$ is a martingale with square integrable differences relative to the natural filtration,
by the Chebyshev inequality,
for all $n \in \N$ and for all real $r>0$,
\[
\Prob\Bigl(|\wt u_n|\ge r\Bigr)\le \frac{\E|\wt u_n|^2}{r^2}
\le \frac{C n}{r^2},
\quad \text{where $C:=\max_{x \in V(G)}\sum_{s\in S}|u(x, s)-\br{v, \zeta_c}|^2\,\mu(s)$}.
\]
By \eqref{eq:lem:martingale},
for all $n \in \N$ and all real $r>0$,
\[
\Prob\Bigl(|\br{v, \Phi(w_n)-n\zeta_c}|\ge r+2\|f\|_\infty\Bigr)
\le \Prob\Bigl(|\wt u_n|\ge r\Bigr)\le \frac{C n}{r^2}.
\]
Applying this to each $v$ in the standard basis in $\R^m$ and taking a union bound show that for a constant $C'>0$, for all $n\in \N$ and for all real $r>0$,
\[
\Prob\Bigl(\|\Phi(w_n)-n \zeta_c\| \ge r\Bigr)\le \frac{C' n}{r^2}.
\]
This shows the claim.
\qed

\subsection{Local CLTs}\label{sec:lclt}

We are now ready to state and prove local CLTs on virtually abelian groups.

For a non-degenerate positive definite (covariance) matrix $\Sigma$ of size $m$,
we consider the multivariate normal distribution of variance $\Sigma$:
\[
\xi_\Sigma(v):=\frac{1}{(2\pi)^{\frac{m}{2}}\sqrt{\det \Sigma}}\exp\Bigl(-\frac{1}{2}\br{v, \Sigma^{-1}v}\Bigr) \quad \text{for $v \in \R^m$}.
\]
Further, we associate to it a probability measure on $\Gamma$, given by
\[
\Nc_{n, \Sigma, \zeta_c}(\gamma):=\pi(\mb{x})\frac{1}{\sum_{w \in \Z^m}\xi_{n\Sigma}(w)}\xi_{n\Sigma}(\Phi(\gamma)-n\zeta_c) \quad \text{for $\gamma=v\mb{x}\in \Gamma=\Lambda \Delta$}.
\]

\begin{theorem}[Local CLT]\label{thm:lclt}
Let $\Gamma=\bigsqcup_{x\in \Delta}\Lambda x$ and $\mu$ be a probability measure on $\Gamma$ with finite second moment such that the support generates the group as a semigroup,
and let $(G, c)$ be the corresponding weighted diagram.
If $\mu$ has period $q$,
then
\[
\left\|q^{-1}\sum_{i=0}^{q-1}\mu_{qn+i}-\Nc_{qn, \Sigma, \zeta_c}\right\|_\TV \to 0 \quad \text{as $n \to \infty$}.
\]
In the above, the covariance matrix $\Sigma$ is computed as 
\[
\br{v_1, \Sigma v_2}=\sum_{e \in E(G)}u_1(e)u_2(e)c(e)-\sum_{e\in E(G)}u_1(e)c(e)\sum_{e\in E(G)}u_2(e)c(e) ,
\]
where $u_i$ denotes the unique harmonic part of $\wh v_i$ for $v_i \in \R^m$, $i=1, 2$.
\end{theorem}

The proof of Theorem \ref{thm:lclt} uses the following estimate.

\begin{proposition}\label{prop:lclt}
In the setting of Theorem \ref{thm:lclt},
if $\mu$ is aperiodic,
then
we have
\begin{equation*}%\label{eq:thm:lclt}
\lim_{n \to \infty}\sup_{\gamma \in \Gamma}
n^{\frac{m}{2}}|\mu_n(\gamma)-\pi([\gamma])\xi_{n\Sigma}(\Phi(\gamma)-n\zeta_c)|=0.
\end{equation*}
\end{proposition}

Before proceeding to the proofs, % of the proposition and theorem, 
we record three useful facts.

\begin{facts}\label{fact:lclt}
We keep the notations as above.
\begin{itemize}
\item[(1)]\label{fit:1}
For each $v \in \R^m$, we have $\wh v=u+d\f$ for some $\f \in C^0(G, \R)$ with a constant $C$ independent of $v$ such that
\begin{equation}\label{eq:thm:potential}
\|\f\|_\infty\le C\|v\|.
\end{equation}
\item[(2)]\label{fit:2}
For every $\e>0$, there exists $\delta>0$ such that for all $v \in D_\delta$,
\begin{equation}\label{eq:thm:betaTaylor}
\beta(v)=2\pi i\langle v,\zeta_c\rangle-2\pi^2\langle v,\Sigma v\rangle+\eta(v) \quad \textrm{where  $|\eta(v)|\le \e\|v\|^2$}.
\end{equation}
\item[(3)]\label{fit:3}
For every $\e_0>0$, for all small enough $\e>0$, the following holds: for all $\delta>0$ and all $n$ large enough (depending on $\delta$),
\begin{equation}\label{eq:thm:errorR}
\int_{D_{\delta\sqrt{n}}}e^{-2\pi^2\langle v,\Sigma v\rangle}\left(e^{\e\|v\|^2}-1\right)dv \le \e_0.
\end{equation}
\end{itemize}
\end{facts}

\proof%[Proof of Facts \ref{fact:lclt}]
For the first point, write $v=\sum_{i=1}^m\alpha_i\mb{e}_i$ with the standard basis $\mb{e}_1, \dots, \mb{e}_m$, and $\alpha_1, \dots, \alpha_m \in \R$. For each $i=1, \dots, m$, taking $\wh{\mb{e}_i}=u_i+d\f_i$, we choose $\f=\sum_{i=1}^m \alpha_i \f_i$, implying \eqref{eq:thm:potential} with  $C=\sqrt{m}\max_{i=1, \dots, m}\|\f_i\|_\infty$ by the Cauchy-Schwarz inequality.

Lemma \ref{lem:pressure} (1) and Lemma \ref{lem:second} 
 show that $(d/dr)|_{r=0}\beta(r v)=2\pi i\dr{c}{\wh v}=2\pi i\langle v,\zeta_c \rangle$,
and Lemma \ref{lem:hessian} shows that the Hessian of $\beta$ at $0$ is $-4\pi^2 \Sigma$.
Since $\beta$ is twice continuously differentiable around $0$, the Taylor theorem implies (\ref{eq:thm:betaTaylor}) with $\eta(v)=o(\|v\|^2)$. There remains to take $\delta$ small enough depending on $\e$ to get the second point.

For the third point, we split the integral according to $0<r<\delta\sqrt{n}$. We have
\[
\int_{D_r}e^{-2\pi^2\langle v,\Sigma v\rangle}\left(e^{\e\|v\|^2}-1\right)dv\le C_\Sigma \sup_{v\in D_r} \left(e^{\e\|v\|^2}-1\right)\le C_\Sigma \e mr^2e^{\e mr^2}.
\]
For $\e>0$ small enough depending on $\Sigma$, we have
\[
\int_{D_{\delta \sqrt{n}}\backslash D_r}e^{-2\pi^2\langle v,\Sigma v\rangle}\left(e^{\e\|v\|^2}-1\right)dv \le \int_{\R^m\backslash D_r}e^{-\pi^2\langle v, \Sigma v \rangle}dv \le Ce^{-cr^2}
\]
for constants $c,C$ depending on $\Sigma$. Choosing $r$ large enough, then $\e$ small enough and finally $n\ge (r/\delta)^2$ we obtain the third point.
\qed

\proof[Proof of Proposition \ref{prop:lclt}]
By the Fourier inversion formula \eqref{eq:fourier_inversion},
we have for $n \in \N$ and for $\gamma \in \Gamma$ with $[\gamma]=x \in V(G)$,
for all $\delta>0$,
\begin{equation}\label{eq:thm:fourier}
\mu_n(\gamma)=\int_{D_\delta}\Lc_{\wh v}^n \1_{x}(x_0)e^{-2\pi i \br{v, \Phi(\gamma)}}\,dv+\int_{D\setminus D_\delta}\Lc_{\wh v}^n \1_{x}(x_0)e^{-2\pi i\br{v, \Phi(\gamma)}}\,dv=: I_1+I_2.
\end{equation}
We will specify $\delta$ later.
If $\mu$ is aperiodic,
then Lemma \ref{lem:aperiodic}
shows that $|I_2| \le C_\delta e^{-c_\delta n}$ for some $c_\delta,C_\delta >0$ and all $n \in \N$.
We will analyze the first term $I_1$. 

Recall that the dominating eigenvalue $e^{\beta(v)}$ of $\Lc_{\wh v}$ depends only on the harmonic part $u$ of $\wh v=u+d\f$ (cf.\ Section \ref{sec:perturbation}).
Since $\Lc_{\wh v}=e^{-2\pi i \f}\Lc_u e^{2\pi i \f}$,
for a small enough $\delta>0$ we have a constant $r_\delta>0$ such that
\[
\Lc_{\wh v}^n \1_{x}(x_0)=e^{n\beta(v)}\br{e^{2\pi i \f}\1_x, f_u}_\pi e^{-2\pi i\f}f_u(x_0)\left(1+O(e^{-r_\delta n})\right) \quad \text{for $v\in D_\delta$},
\]
where $f_u$ denotes the normalized eigenfunction of $\Lc_u$ for the eigenvalue $e^{\beta(v)}$.
By Lemma~\ref{lem:pressure} \eqref{eq:lem:eigen},
we have
$f_{u}(x)=1+O\left(\|v\|^2\right)$
and so
\[
\br{e^{2\pi i\f}\1_x, f_u}_\pi e^{-2\pi i\f}f_u(x_0)
=
\pi(x)e^{2\pi i(\f(x)-\f(x_0))}\left(1+O\left(\|v\|^2\right)\right).
\]
Then we can express
\[
I_1=\pi(x)\int_{D_\delta} e^{n\beta(v)}e^{2\pi i(\f(x)-\f(x_0))}e^{-2\pi i\langle v, \Phi(\gamma)\rangle}\left(1+O(\|v\|^2)+O(e^{-r_\delta n})\right)dv
\]
The change of variable $v \mapsto v/\sqrt{n}$ yields $u\mapsto u/\sqrt{n}$ and $\f\mapsto \f/\sqrt{n}$.
Note that since $\f$ is real-valued, by \eqref{eq:thm:potential} we have
\[
e^{2\pi i (\f(x)-\f(x_0))/\sqrt{n}}=1+O\left(\frac{\|v\|}{\sqrt{n}}\right).
\]
Using also the expression for $\beta(v)$ given in \eqref{eq:thm:betaTaylor}, we obtain
\[
I_1=\frac{\pi(x)}{n^{\frac{m}{2}}}\int_{D_{\delta \sqrt{n}}} e^{\sqrt{n}2\pi i \langle v,\zeta_c\rangle-2\pi^2\langle v,\Sigma v\rangle+n\eta(\frac{v}{\sqrt{n}})} e^{-2\pi i\langle \frac{v}{\sqrt{n}},\Phi(\gamma)\rangle}\left(1+O\left(\frac{\|v\|}{\sqrt{n}}\right)+O(e^{-r_\delta n})\right)dv
\]
Let $\e>0$ and $\delta>0$ be small enough (depending on $\Sigma$) such that \eqref{eq:thm:betaTaylor} gives 
\begin{equation}\label{eq:betaepsilon}
-2\pi^2\langle v, \Sigma v\rangle +|\eta(v)| \le -\pi^2\langle v, \Sigma v\rangle, \quad \text{for all $v \in D_\delta$}.
\end{equation} 
Then the integral of the modulus for the terms $O\left(\|v\|/\sqrt{n}\right)+O(e^{-r_\delta n})$ in the above expression is bounded above by
\[
\int_{D_{\delta\sqrt{n}}} e^{-\pi^2\langle v,\Sigma v\rangle}\left( \frac{\|v\|}{\sqrt{n}}+e^{-r_\delta n}\right) dv=O\left(\frac{1}{\sqrt{n}}\right).
\]
Therefore we get, setting $R_n(v):=e^{n\eta(\frac{v}{\sqrt{n}})}-1$,
\[
I_1=\frac{\pi(x)}{n^{\frac{m}{2}}}\int_{D_{\delta \sqrt{n}}} e^{2\pi i\left\langle v,\sqrt{n}\zeta_c-\frac{\Phi(\gamma)}{\sqrt{n}}\right\rangle} e^{-2\pi^2\langle v,\Sigma v\rangle}\left(1+R_n(v)\right) dv +O\left(\frac{1}{n^{\frac{m+1}{2}}}\right).
\]
Now by \eqref{eq:thm:errorR}, provided $\e>0$ is small enough, the integral of the term with $R_n(v)$ can be made smaller than any $\e_0>0$. Coming back to \eqref{eq:thm:fourier}, we have proved
\[
\left| \mu_n(\gamma)-\frac{\pi(x)}{n^{\frac{m}{2}}}\int_{D_{\delta\sqrt{n}}} e^{2\pi i\br{v, \sqrt{n}\zeta_c-\frac{\Phi(\gamma)}{\sqrt{n}}}}e^{-2\pi^2\br{v, \Sigma v}}\,dv\right| \le \frac{\pi(x)\e_0}{n^{\frac{m}{2}}}+O\left(\frac{1}{n^{\frac{m+1}{2}}} \right)
\]
For a given $\e_0>0$, we choose $\e>0$ and $\delta>0$ such that \eqref{eq:thm:betaTaylor}, \eqref{eq:thm:errorR}, and \eqref{eq:betaepsilon} hold.

On the other hand, the Fourier transform directly shows that for $w \in \R^m$,
\[
\xi_{n\Sigma}(w-n\zeta_c)=\frac{1}{n^{\frac{m}{2}}}\int_{\R^m}e^{2\pi i\br{v, \sqrt{n}\zeta_c-\frac{w}{\sqrt{n}}}}e^{-2\pi^2\br{v, \Sigma v}}\,dv.
\]
There exists a constant $c_\delta'>0$, depending only on $\Sigma$, such that for all $\gamma \in \Gamma$ and all $n\ge 1$,
\[
\xi_{n\Sigma}(\Phi(\gamma)-n\zeta_c)
=
\frac{1}{n^{\frac{m}{2}}}\int_{D_{\delta \sqrt{n}}}
e^{2\pi i\br{v, \sqrt{n}\zeta_c-\frac{\Phi(\gamma)}{\sqrt{n}}}}e^{-2\pi^2\br{v, \Sigma v}}\,dv+O(e^{-c_\delta' n}).
\]

Summarizing the above estimates shows that for every $\e_0>0$ there exists a constant $C_{\e_0}$ such that for all large enough $n$, uniformly in $\gamma \in \Gamma$ with $[\gamma]=x$,
\begin{align*}
n^{\frac{m}{2}}|\mu_n(\gamma)-\pi(x)\xi_{n\Sigma}(\Phi(\gamma)-n\zeta_c)|
 \le \pi(x)\e_0+C_{\e_0}\left(n^{-\frac{1}{2}}+ n^{\frac{m}{2}}e^{-c_\delta' n}\right).
\end{align*}
This shows the proposition.
\qed

\proof[Proof of Theorem \ref{thm:lclt}]
Note that the covariance matrix $\Sigma$ is computed in Lemma \ref{lem:hessian}.
First we show the claim assuming that $\mu$ is aperiodic.
By Lemma \ref{lem:martingale}, for a constant $C$ for all $n\in \N$ and all $r>0$,
\begin{equation}\label{eq:thm:lclt:martingale}
\Prob\left(\|\Phi(w_n)-n\zeta_c\|\ge r\right)\le \frac{C n}{r^2}.
\end{equation}
Moreover,
a direct computation yields a constant $C'$ such that for all $n\ge 1$ and all $r>0$,
\begin{equation*}
\sum_{v \in \Z^m, \|v\|\ge r}\xi_{n \Sigma}(v)\le C' \exp\left(-\frac{r^2}{C'n}\right).
\end{equation*}
Hence we have
\begin{equation}\label{eq:thm:lclt:gaussian}
\sum_{\gamma \in \Gamma, \|\Phi(\gamma)-n\zeta_c\|\ge r}\pi([\gamma])\xi_{n\Sigma}(\Phi(\gamma)-n\zeta_c) \le C'\exp\left(-\frac{r^2}{C' n}\right) \le \frac{C'^2n}{r^2}.
\end{equation}
By Proposition \ref{prop:lclt}, \eqref{eq:thm:lclt:martingale} and \eqref{eq:thm:lclt:gaussian}, for $C'':=C+C'^2$, for every $\e>0$, for all large enough $n$, and for all $r>0$,
\begin{align*}
\sum_{\gamma \in \Gamma}|\mu_n(\gamma)-\pi([\gamma])\xi_{n\Sigma}(\Phi(\gamma)-n\zeta_c)|
\le
\#B(r)\frac{\e}{n^{\frac{m}{2}}}+\frac{C''n}{r^2}.
\end{align*}
The inequality follows by dividing the sum over $\gamma$ into the two parts where $\|\Phi(\gamma)-n\zeta_c\|$ is less than $r$ or at least $r$.
In the above, $\#B(r)$ denotes the number of $\gamma \in \Gamma$ with $\|\Phi(\gamma)-n\zeta_c\|<r$,
and $\#B(r)\le C_0 r^m$ for a constant $C_0$ independent of $n$ or $r$.
Thus, letting $C:=C_0+C''$, we obtain for $r=\e^{-\frac{1}{m+2}}\sqrt{n}$,
\[
\sum_{\gamma \in \Gamma}|\mu_n(\gamma)-\pi([\gamma])\xi_{n\Sigma}(\Phi(\gamma)-n\zeta_c)|
\le \frac{C_0 \e r^m}{n^\frac{m}{2}}+\frac{C''n}{r^2}
=C_0\e^\frac{2}{m+2}+C''\e^{\frac{2}{m+2}}=C\e^{\frac{2}{m+2}}.
\]
Furthermore, by the Poisson summation formula, for a constant $c>0$, for all $n\ge 1$,
\[
\sum_{v \in \Z^m}\xi_{n\Sigma}(v)=\sum_{v \in \Z^m}e^{-2\pi^2 n\br{v, \Sigma v}}=1+\sum_{v \in \Z^m\setminus \{0\}}e^{-2\pi^2 n\br{v, \Sigma v}}=1+O(n^{-1}e^{-cn}).
\]
By choosing  a bigger constant $C$, for every $\e>0$, for all large enough $n$,
\[
\|\mu_n-\Nc_{n, \Sigma, \zeta_c}\|_\TV=\frac{1}{2}\sum_{\gamma \in \Gamma}|\mu_n(\gamma)-\Nc_{n, \Sigma, \zeta_c}(\gamma)|\le C\e^{\frac{2}{m+2}}.
\]
This shows the claim if $\mu$ is aperiodic.

In general, if $\mu$ has period $q$, then $\mu_q$ is aperiodic.
Applying the above discussion to $\mu_q$ with the initial distribution $q^{-1}(\delta_{\id}+\mu+\cdots+\mu_{q-1})$ yields
\[
\Bigl\|q^{-1}\sum_{i=0}^{q-1}\mu_{qn+i}-q^{-1}\sum_{i=0}^{q-1}\Nc_{qn+i, \Sigma, \zeta_c}\Bigr\|_\TV \to 0 \quad \text{as $n \to \infty$}.
\]
Noting that for each $i=0, \dots, q-1$,
\[
\Bigl\|\Nc_{qn+i, \Sigma, \zeta_c}-\Nc_{qn, \Sigma, \zeta_c}\Bigr\|_\TV =O\Bigl(\frac{1}{\sqrt{n}}\Bigr),
\]
we conclude the proof.
\qed

\begin{remark}\label{rem:quantitative}
In Theorem \ref{thm:lclt},
if $\mu$ has finite exponential moment (e.g.\ finite support), is symmetric and aperiodic,
the proof is simplified and strengthened.
In this case, $\zeta_c=0$, the period $q=1$, and $\Prob(\|\Phi(w_n)\|\ge r) \le C\exp(-r^2/(Cn))$ for $n\ge 1$ and $r>0$ in Lemma \ref{lem:martingale}.
We have
\[
\left\|\mu_n-\Nc_{n, \Sigma, 0}\right\|_\TV=O\left(\frac{(\log n)^{m/2}}{n^{1/2}}\right),
\]
as $n\to \infty$.
This follows since $\e_0$ can be chosen $O(n^{-1})$ in the proof of Proposition~\ref{prop:lclt} (using Lemma \ref{lem:pressure} (4) and the Taylor theorem on $\beta(u)$ up to the fourth order) and $r=A\sqrt{n\log n}$ for a large enough constant $A$ in the proof of Theorem \ref{thm:lclt}.
This is done in full details in~\cite{t-ns} for affine Weyl groups.
We refrain from reproducing the whole argument in the virtually abelian case.
\end{remark}

\begin{remark}\label{rem:ks}
The local CLT in the form of Theorem \ref{thm:lclt} has been shown, cf.\ \cite[Theorem 5.2]{ks}.
See also history and background therein for random walks with internal degrees of freedom.
We have provided a proof adapted to the present setting with an explicit form of covariance.
\end{remark}

\section{Noise sensitivity problem}\label{sec:ns}

We discuss noise sensitivity problem on random walks on $\Gamma$
by applying the local CLT established in Theorem \ref{thm:lclt}.

In fact, we first consider a more general setting.
For $i=1, 2$,
let $\Gamma_i$ be finitely generated infinite virtually abelian groups,
and $\Lambda_i$ be finite index normal subgroups of $\Gamma_i$ isomorphic to $\Z^{m_i}$ for some $m_i\ge 1$.
Further, let $\nu$ be a probability measure on $\Gamma_1\times \Gamma_2$. Denote $\mu^{(1)}$ and $\mu^{(2)}$ respectively the pushforwards to each factor.
We assume that $\nu$ is aperiodic and the support of $\nu$ generates $\Gamma_1\times \Gamma_2$ as a semigroup.
Note that $\mu^{(i)}$ is aperiodic and the support of $\mu^{(i)}$ generates $\Gamma_i$ as a semigroup for $i=1 ,2$.

Furthermore, 
$\nu$ has finite second moment if and only if both $\mu^{(i)}$ have finite second moment.
For $\Lambda_i$-equivariant map $\Phi_i:\Gamma \to \Lambda_i= \Z^{m_i}$,
let 
\[
\Phi: \Gamma_1 \times \Gamma_2 \to \Z^{m_1}\times \Z^{m_2}
\]
be the product map of $\Phi_1$ and $\Phi_2$.
This map $\Phi$ is $\Lambda_1\times \Lambda_2$-equivariant.
Let $(G, c)$ be the corresponding weighted diagram associated with $(\Gamma_1 \times \Gamma_2, \nu)$ and $\Lambda_1 \times \Lambda_2$.
We apply to $\Gamma_1 \times \Gamma_2$ and $\nu$ the whole discussion we have developed so far.

\subsection{The structure of covariance matrix related to transfer homomorphisms}\label{sec:structure} We apply the local CLT to random walks on a product of virtually abelian groups.

As in Section~\ref{sec:pre}, for $i=1,2$, we denote by $\vartheta_i:\Gamma_i \to \Lambda_i$ the transfer homomorphism, fix an identification $\Lambda_i$ with $\Z^{m_i}$ in $\R^{m_i}$, and consider the orthogonal decompositions:
\[
\R^{m_i}=\im (\uth_i) \oplus \im(\uth_i)^\perp,
\]
where $\uth_i$ is the linear operator on $\mathbb{R}^{m_i}$ which coincides with $\vartheta_i/\#F_i$ on $\Lambda_i$.

\begin{theorem}[Structure Theorem]\label{thm:structure}
Let $\nu$ be an aperiodic finite second moment probability measure on $\Gamma_1\times \Gamma_2$, with marginals $\mu^{(i)}$, and we keep the notations above. Then $\|\nu_n -\Nc_{n, \Sigma_\nu, \zeta_c}\|_{\TV} \to 0$ as $n \to \infty$ with $\zeta_c \in V_{\vartheta_1}\oplus V_{\vartheta_2}$ and the covariance matrix of the form
\[
\Sigma_\nu=
\left(\begin{array}{c|c}  \Sigma_{\mu^{(1)}} & 
\begin{matrix} \ast & 0 \\
0 & 0 \end{matrix}
 \\
 \hline
\begin{matrix} \ast & 0 \\
0 & 0 \end{matrix}
 &  \Sigma_{\mu^{(2)}} 
\end{array}\right)
\]
in a basis adapted to the decomposition $\R^{m_1}\times \R^{m_2}=\im (\uth_1) \oplus \im(\uth_1)^\perp \oplus \im (\uth_2) \oplus \im(\uth_2)^\perp$. 
\end{theorem}

In particular, when both $\Gamma_i$ admit no nonzero homomorphisms to $\Z$, the drift is zero. This is known by the general result of Karlsson and Ledrappier~\cite[Corollary 4]{Karlsson-Ledrappier}: Liouville  random walks on groups with no nonzero homomorphisms to $\Z$ have zero drift.

Prior to the proof, we record an important fact about harmonic decompositions in products. Recall that the edge set $E(G^{(1)}\times G^{(2)})$ is identified with $\Delta_1 \times S^{(1)}\times \Delta_2 \times S^{(2)}$. Edges will be simply denoted as tuples $(x_1,s_1,x_2,s_2)$.
The quotient diagram is denoted $\mb{F}:=(\Lambda_1\times \Lambda_2)\backslash (\Gamma_1\times \Gamma_2)=F_1\times F_2$.

\begin{fact}\label{fact:u1}
In the setting of Theorem~\ref{thm:structure}, for $v_1 \in \R^{m_1}$, we let $\mb{v}_1=(v_1,0) \in \R^{m_1}\times \R^{m_2}$. Denote their harmonic decompositions as
\[
\wh{v}_1=u_1+df_1 \quad \textrm{and} \quad \wh{\mb{v}}_1=\mb{u}_1+d\mb{f}_1.
\]
Then for all $(x_1,s_1,x_2,s_2) \in \Delta_1 \times S^{(1)}\times \Delta_2 \times S^{(2)}$,
\[
\mb{u}_1(x_1,s_1,x_2,s_2)=u_1(x_1,s_1) \quad \textrm{and} \quad \mb{f}_1(x_1,x_2)=f_1(x_1).
\]
\end{fact}

\proof%[Proof of Fact~\ref{fact:u1}]
It is enough to check the decomposition
\begin{align*}
\wh{\mb{v}}_1(x_1,s_1,x_2,s_2)&=\langle \mb{v}_1,\Phi_{(x_1,s_1,x_2,s_2)} \rangle =\langle (v_1,0),(\Phi^{(1)}_{(x_1,s_1)},\Phi^{(2)}_{(x_2,s_2)}) \rangle \\
&=\langle v_1, \Phi^{(1)}_{(x_1,s_1)}\rangle = u_1(x_1,s_1)+df_1(x_1,s_1)
\end{align*}
and the harmonicity of $(x_1,s_1,x_2,s_2) \mapsto u_1(x_1,s_1)=:\mb{u}_1(x_1,s_1,x_2,s_2)$. Indeed, note that 
\begin{equation}\label{eq:conductance}
\mb{c}(x_1,s_1,x_2,s_2)=\pi(x_1,x_2)\nu(s_1,s_2)=\frac{1}{\#\mb{F}}\nu(s_1,s_2).
\end{equation} 
%and $\dr{c}{u_1}=0$ by Lemma \ref{lem:diagonal} as $v_1\in W_{\vartheta_1}$.
Thus, we have
\begin{align*}
d^*\mb{u}_1(x_1,x_2)&= -\sum_{e:oe=(x_1,x_2)}\frac{\mb{c}(e)}{\pi(x_1,x_2)} \mb{u}_1(x_1,s_1,x_2,s_2)
=-\sum_{s_1,s_2} \nu(s_1,s_2)u_1(x_1,s_1)  \\
&=-\sum_{s_1} \mu^{(1)}(s_1)u_1(x_1,s_1)=d^*u_1(x_1)
\end{align*}
and 
\begin{align}\label{eq:chicu}
\chi_{\mb{c}}(\mb{u}_1) &= \sum_{e \in E(G)} \mb{u}_1(e)\mb{c}(e) =\frac{1}{\#\mb{F}}\sum_{x_1,s_1,x_2,s_2} u_1(x_1,s_1)\nu(s_1,s_2) \nonumber \\
&=\frac{1}{\#\mb{F}}\sum_{s_1,s_1} u_1(x_1,s_1)\mu^{(1)}(s_1)\#F_2 =\chi_c(u_1).
\end{align}
Then $d^*\mb{u}_1(x_1,x_2)+\chi_{\mb{c}}(\mb{u}_1)=d^*u_1(x_1)+\chi_c(u_1)=0$ as $u_1$ is harmonic for $\mu^{(1)}$.
\qed

\proof[Proof of Theorem~\ref{thm:structure}]
The convergence is given by the local CLT (Theorem \ref{thm:lclt}). 
Letting $\alpha_i$ denote the cocycle in the definition of $\vartheta_i:\Gamma_i \to \Lambda_i$ for $i=1, 2$,
we define $\alpha(x_1,x_2,s_1,s_2)=(\alpha_1(x_1,s_1), \alpha_2(x_2,s_2))$.
Using \eqref{eq:conductance}, the drift is computed as
\begin{align*}
\zeta_{\mb{c}}&=\sum_{e\in E(G)}\mb{c}(e)\Phi_e=\sum_{x_1,s_1,x_2,s_2}\frac{1}{\#\mb{F}}\nu(s_1,s_2)\alpha(x_1,s_1,x_2,s_2) \\
&=\frac{1}{\#\mb{F}}\sum_{s_1,s_2} \nu(s_1,s_2)\left(\sum_{x_2, x_1}\alpha_1(x_1,s_1), \sum_{x_1, x_2}\alpha_2(x_2,s_2) \right).
\end{align*}
It belongs to $\mathrm{im}(\uth_1)\oplus  \mathrm{im}(\uth_2)$ since $\sum_{x_i}\alpha_i(x_i,s_i)=\vartheta_i(s_i)$.

To get the zeros in the covariance matrix, consider $\mb{v}_1=(v_1, 0)$ for $v_1 \in \mathrm{im}(\uth_1)^\perp \subset \R^{m_1}$ and $\mb{v}_2=(0, v_2)$ for $v_2 \in \R^{m_2}$ and compute, using the explicit form in Theorem~\ref{thm:lclt}, Fact~\ref{fact:u1}, and  equation (\ref{eq:chicu})
\begin{align}\label{eq:cov}
\langle \mb{v}_1, \Sigma_\nu \mb{v}_2\rangle &= \sum_{e\in E(G)}\mb{u}_1(e)\mb{u}_2(e) \mb{c}(e)-\left(\sum_{e\in E(G)}\mb{u}_1(e)\mb{c}(e)\right)\left(\sum_{e\in E(G)}\mb{u}_2(e)\mb{c}(e)\right) \nonumber\\
&= \sum_{x_1, x_2,s_1, s_2}u_1(x_1, s_1)u_2(x_2, s_2)\frac{\nu(s_1,s_2)}{\#\mb{F}} -\chi_{{c}}({{u}_1})\chi_{{c}}({{u}_2}). 
\end{align}
Reordering the sums,
we have
\[
\sum_{x_1, x_2,s_1, s_2}u_1(x_1, s_1)u_2(x_2, s_2)\frac{\nu(s_1,s_2)}{\#\mb{F}}=\sum_{x_2,s_1,s_2} u_2(x_2,s_2)\frac{\nu(s_1,s_2)}{\#\mb{F}}\sum_{x_1} u_1(x_1,s_1)=0,
\]
where the last equality follows from Lemma \ref{lem:diagonal} as $v_1 \in \mathrm{im}(\uth_1)^\perp$. Moreover $\chi_{{c}}({{u}_1})=0$ by Lemma~\ref{lem:diagonal}.
The same computation holds exchanging the roles of first and second group factors.
\qed

\medskip

The following corollary of Theorem \ref{thm:structure} gives Theorem~\ref{thm:zerohom_intro} from the introduction when the groups are infinite.

\begin{corollary}\label{thm:zerohom}
For $i=1, 2$, let $\Gamma_i$ be a finitely generated infinite virtually abelian group,
and $\nu$ be an aperiodic probability measure on $\Gamma_1\times \Gamma_2$ with finite second moment such that the support of $\nu$ generates $\Gamma_1\times \Gamma_2$ as a semigroup and $\nu$ has the marginal $\mu^{(i)}$ on $\Gamma_i$.
If~$\Gamma_1$ admits no nonzero homomorphisms onto $\Z$,
then 
\[
\big\|\nu_n-\mu^{(1)}_n\times \mu^{(2)}_n\big\|_\TV \to 0 \quad \text{as $n \to \infty$}.
\]
\end{corollary}

\proof%[Proof of Corollary~\ref{thm:zerohom}]
The transfer homomorphism of $\Gamma_1$ vanishes so the vector spaces $\mathrm{Im}(\uth_1)$ is zero. Theorem~\ref{thm:structure} shows that $\Sigma_\nu$ is a block diagonal symmetric matrix along the decomposition $\R^{m_1} \times \R^{m_2}$.
Two blocks are the covariance matrices for $\Sigma_{\mu^{(i)}}$. The other blocks are only zero. Therefore, $\Sigma_\nu$ coincides with the covariance matrix for the $(\mu^{(1)}\times \mu^{(2)})$-random walk on $\Gamma_1\times \Gamma_2$.
The local CLT (Theorem~\ref{thm:lclt}) applied to $\nu$ and $\mu^{(1)}\times \mu^{(2)}$
yields by the triangle inequality,
\[
\big\|\nu_n-\mu^{(1)}_n\times\mu^{(2)}_n\big\|_\TV \le \big\|\nu_n-\Nc_{n, \Sigma_\nu, 0}\big\|_\TV+\big\|\mu^{(1)}_n\times \mu^{(2)}_n-\Nc_{n, \Sigma_\nu, 0}\big\|_\TV \to 0 \quad \text{as $n \to \infty$}.
\]
This concludes the claim.
\qed

\proof[Proof of Theorem~\ref{thm:zerohom_intro}]
First, in the case when both groups are infinite, we apply Corollary~\ref{thm:zerohom}.
Next, in the case when one group $\Gamma$ is infinite and the other group is finite $G$, we apply Theorem \ref{thm:lclt} to $\Gamma \times G$.
Note that $\Nc_{n, \Sigma, \zeta_c}$ is defined as the product of the uniform distribution on $F \times G$ and the discrete Gaussian distribution on $\Lambda$, where $\Lambda\backslash \Gamma=F$.
Finally, in the case when both groups are finite, both $\nu_n$ and $\mu_n^{(1)}\times \mu_n^{(2)}$ tend to the uniform distribution on the product group.
\qed

\subsection{Applications to noise sensitivity}

Recall that noise sensitivity of the $\mu$-random walk on a group $\Gamma$ is concerned with the measures $\pi^\rho=\rho (\mu\times \mu)+(1-\rho)\mu_\diag$ for $\rho \in [0,1]$ on $\Gamma\times\Gamma$.
As a special case of Corollary~\ref{thm:zerohom}, we obtain the following, which is the main step towards Theorem~\ref{thm:intro}.

\begin{corollary}\label{cor:ns}
Let $\Gamma$ be a finitely generated infinite virtually abelian group and $\mu$ be an aperiodic probability measure on $\Gamma$ with finite second moment such that the support generates the group as a semigroup.
If $\Gamma$ admits no nonzero homomorphism onto $\Z$,
then
\[
\left\|\pi^\rho_n-\mu_n\times \mu_n\right\|_\TV \to 0 \quad \text{as $n\to \infty$},
\]
for all $\rho \in (0, 1]$, i.e.\ the $\mu$-random walk on $\Gamma$ is noise sensitive in total variation.
\end{corollary}

\proof
For each $\rho \in (0, 1]$, we have $\supp \pi^\rho=\supp \mu\times \supp \mu$, and the support of $\pi^\rho$ generates $\Gamma\times \Gamma$ as a semigroup since $\mu$ is aperiodic and $\supp \mu$ generates $\Gamma$ as a semigroup by assumption.
Furthermore, $\pi^\rho$ is aperiodic since $\mu$ is aperiodic, and $\pi^\rho$ has finite second moment since $\mu$ has finite second moment by assumption.
Applying Corollary~\ref{thm:zerohom} with $\nu=\pi^\rho$ for each $\rho \in (0, 1]$, $\Gamma_i=\Gamma$, and $\mu_i=\mu$ for $i=1, 2$ yields the claim.
\qed

\medskip

The local CLT implies the following corollary.

\begin{corollary}\label{cor:rho1}
Let $\Gamma$ be a finitely generated virtually abelian group and $\mu$ be a finite second moment probability measure on $\Gamma$ such that the support generates the group as a semigroup.
Then
\[
\lim_{\rho \to 1}\limsup_{n \to \infty}\left\|\pi^\rho_n-\mu_n\times \mu_n\right\|_\TV=0 
\]
\end{corollary}

\proof
Note that if $\mu$ has period $q$, then $\pi^\rho$ has period $q$ for all $\rho \in (0, 1]$ since $\supp \pi^\rho=\supp \mu\times \supp \mu$.
Thus, $\pi^\rho_{nq+i}$ and $\mu_{qn+i}\times \mu_{qn+i}$ have the common support, and $\pi^\rho_{nq+i}$, $i=0, 1, \dots, q-1$, have disjoint supports.
Therefore we have
\[\sum_{i=0}^{q-1}\left\|\pi^\rho_{nq+i}-\mu_{qn+i}\times \mu_{qn+i}\right\|_\TV
=
\left\|\sum_{i=0}^{q-1}\pi^\rho_{nq+i}-\sum_{i=0}^{q-1}\mu_{qn+i}\times \mu_{qn+i}\right\|_\TV.
\]
For each $\rho \in (0, 1]$,
the local CLT (Theorem \ref{thm:lclt}) implies that 
\begin{equation}\label{eq:thm:nonzerohom:lclt}
\left\|\sum_{i=0}^{q-1}\pi^\rho_{qn+i}-q\Nc_{qn, \Sigma^\rho, \zeta_c}\right\|_\TV \to 0 \quad \text{as $n \to \infty$}.
\end{equation}
In the above, $\zeta_c=\sum_{e \in E(G)}c(e)\Phi_e \in \Lambda\times \Lambda$ has the same component in each factor $\Lambda$, in particular, independent of $\rho \in (0, 1]$ since $\pi^\rho$ has marginals $\mu$ in both factors. 
Note that
\[
\left\|\Nc_{n, \Sigma^\rho, \zeta_c}-\Nc_{n, \Sigma^1, \zeta_c}\right\|_\TV=\left\|\Nc_{n, \Sigma^\rho, 0}-\Nc_{n, \Sigma^1, 0}\right\|_\TV.
\]
Furthermore, we have
\begin{equation}\label{eq:limit_gaussian}
\lim_{n\to \infty}\left\|\Nc_{n, \Sigma^\rho, 0}-\Nc_{n, \Sigma^1, 0}\right\|_\TV=\frac{1}{2}\|\xi_{\Sigma^\rho}-\xi_{\Sigma^1}\|_{L^1(\R^{2m})}.
\end{equation}
Before we show \eqref{eq:limit_gaussian},
let us conclude the claim.
By \eqref{eq:limit_gaussian}, the triangle inequality, and \eqref{eq:thm:nonzerohom:lclt} for the case when $\rho \in (0, 1]$ and for the case when $\rho=1$,
we have
\begin{align}\label{eq:thm:nonzerohom:n}
\lim_{n \to \infty}q^{-1}\left\|\sum_{i=0}^{q-1}\pi^\rho_{nq+i}-\sum_{i=0}^{q-1}\mu_{qn+i}\times \mu_{qn+i}\right\|_\TV
&=\lim_{n\to \infty}\left\|\Nc_{qn, \Sigma^\rho, 0}-\Nc_{qn, \Sigma^1, 0}\right\|_\TV \nonumber\\
&=\frac{1}{2}\|\xi_{\Sigma^\rho}-\xi_{\Sigma^1}\|_{L^1(\R^{2m})}.
\end{align}
Furthermore, $\Sigma^\rho\to \Sigma^1$ as $\rho \to 1$ by Remark~\ref{rem:continuity} (see also the explicit form of covariance matrix provided by Proposition~\ref{thm:structurepi} in the case of semi-direct products).
This implies that the right hand side of \eqref{eq:thm:nonzerohom:n} tends to $0$ as $\rho \to 1$.

Let us show \eqref{eq:limit_gaussian}.
By the definition of $\Nc_{n, \Sigma^\rho, 0}$ (see, the one above Theorem \ref{thm:lclt}), there exists a constant $c>0$ such that for all large enough $n$,
\[
\|\Nc_{n, \Sigma^\rho, 0}-\Nc_{n, \Sigma^1, 0}\|_\TV
=\frac{1}{2}\sum_{x \in \Delta^2}\pi(x)\sum_{v \in \Z^{2m}}|\xi_{n\Sigma^\rho}(v)-\xi_{n\Sigma^1}(v)|+O(e^{-c n}).
\]
Note that there exists a constant $c>0$ such that for all real $\lambda>1$ and all integer $n\ge 1$, 
\[
\sum_{\|v\|> \lambda n^{1/2}}|\xi_{n\Sigma^\rho}(v)-\xi_{n\Sigma^1}(v)|=O(e^{-c\lambda^2})
\quad \text{and} \quad
\int_{\|v\| > \lambda n^{1/2}}|\xi_{n\Sigma^\rho}(v)-\xi_{n\Sigma^1}(v)|\,dv=O(e^{-c\lambda^2}).
\]
Furthermore, by estimating the gradients of the Gaussians, we have
\begin{align*}
\sum_{\|v\| \le \lambda n^{1/2}}|\xi_{n\Sigma^\rho}(v)-\xi_{n\Sigma^1}(v)|
=\int_{\|v\| \le \lambda n^{1/2}}|\xi_{n\Sigma^\rho}(v)-\xi_{n\Sigma^1}(v)|\,dv+O\left(\frac{1}{\sqrt{n}}\right).
\end{align*}
By change of variables $v\mapsto v/\sqrt{n}$, we have
\[
\int_{\R^{2m}}|\xi_{n\Sigma^\rho}(v)-\xi_{n\Sigma^1}(v)|\,dv=\int_{\R^{2m}}|\xi_{\Sigma^\rho}(v)-\xi_{\Sigma^1}(v)|\,dv=\|\xi_{\Sigma^\rho}-\xi_{\Sigma^1}\|_{L^1(\R^{2m})}.
\]
Summarizing the above estimates, first $n \to \infty$, and then $\lambda \to \infty$, we obtain \eqref{eq:limit_gaussian}.
\qed

\begin{corollary}\label{thm:nonzerohom}
Let $\Gamma$ be a finitely generated infinite virtually abelian group and $\mu$ be a finite second moment probability measure on $\Gamma$ such that the support generates the group as a semigroup.
If moreover $\Gamma$ admits a nonzero homomorphism onto $\Z$,
then
\[
\lim_{\rho \to 0}\liminf_{n \to \infty}\left\|\pi^\rho_n-\mu_n\times \mu_n\right\|_\TV=1,
\]
in particular, the $\mu$-random walk on $\Gamma$ is not noise sensitive in total variation.
\end{corollary}

\proof
Considering the pushforward of $\pi^\rho$ from $\Gamma\times \Gamma$ to $\im(\vartheta)\times\im(\vartheta)$ via the product of transfer homomorphisms.
We consider the random walk on the finitely generated torsion free abelian group $\im(\vartheta)\times\im(\vartheta)$.
Note that the pushforward of $\pi^\rho$ has finite second moment, and that the projection does not increase total variation distances.
By \cite[Theorem A.1]{t-ns}, we have the first claim.

The second claim follows from the first by the definition of noise sensitivity.
\qed

\proof[Proof of Theorem \ref{thm:nonzerohom_intro}]
It follows from Corollary~\ref{cor:rho1} and Corollary~\ref{thm:nonzerohom}.
\qed

\proof[Proof of Theorem \ref{thm:intro}]
If $\Gamma$ is finite, then the $\mu$-random walk is noise sensitive in total variation \cite[Proposition 5.1]{bb}.
If $\Gamma$ is infinite, then the claim follows from Corollary \ref{cor:ns} and Corollary~\ref{thm:nonzerohom}.
\qed

\subsection{An explicit form of the structure theorem for semi-direct product groups}\label{sec:semi}
In the particular case of semi-direct product groups, Theorem~\ref{thm:structure} is more explicit. We use the notations above with $\Gamma_1=\Gamma_2=\Gamma$.

\begin{proposition}\label{thm:structurepi}
Assume that the virtually abelian group is a semi-direct product $\Gamma=\Lambda \rtimes F$. Let  $\mu$ be a finite second moment probability measure on $\Gamma$ such that the support generates the group as a semigroup. For the measure $\pi^\rho=\rho (\mu\times \mu)+(1-\rho)\mu_\diag$ on $\Gamma\times\Gamma$, we have
\[
\Sigma_{\pi^\rho}=
\left(\begin{array}{c|c} \Sigma_{\mu}    & 
\begin{matrix} (1-\rho)\Sigma_{\mu}|_{\mathrm{im}(\uth)\times \mathrm{im}(\uth)} & 0 \\
0 & 0 \end{matrix}
 \\
 \hline
\begin{matrix} (1-\rho)\Sigma_{\mu}|_{\mathrm{im}(\uth)\times \mathrm{im}(\uth)} & 0 \\
0 & 0 \end{matrix}
 &  \Sigma_{\mu} 
\end{array}\right)
\]
in a basis adapted to the decomposition $\R^{m_1}\times \R^{m_2}=\im (\uth_1) \oplus \im(\uth_1)^\perp \oplus \im (\uth_2) \oplus \im(\uth_2)^\perp$.
\end{proposition}

This proposition implies in particular that, regarding noise sensitivity, the underlying Euclidean space $\R^m$ can be decomposed into two complementary subspaces. On one of these subspaces, namely $\mathrm{im}(\uth)$, the effect of noise sensitivity is the same as on free abelian groups.
On the other subspace, that is $\mathrm{im}(\uth)^\perp$, the random walk behaves as noise sensitive.

Note also that as $\uth$ is a projection, we have
\[
 \uth^{\ast} \Sigma_\mu \uth=\left( \begin{matrix} \Sigma_{\mu}|_{\mathrm{im}(\uth)\times \mathrm{im}(\uth)} & 0 \\
0 & 0 \end{matrix}\right), \quad \textrm{so} \quad \Sigma_{\pi^\rho}=  \left( \begin{matrix} \Sigma_{\mu} & (1-\rho)\uth^{\ast} \Sigma_\mu \uth \\
(1-\rho)\uth^{\ast} \Sigma_\mu \uth & \Sigma_{\mu} \end{matrix}\right).
\]

\begin{proof}[Proof of Proposition \ref{thm:structurepi}]
Let $\mb{v}_i \in \mathrm{im}(\uth_i)$, for $i=1,2$. Using the definitions of $\pi^\rho$ and $\chi_c(u_i)=\sum_{x,s}u_i(x,s)\mu(s)/\#F$, we compute from equation (\ref{eq:cov})
\begin{align*}
\langle \mb{v}_1, \Sigma_\nu \mb{v}_2\rangle &=\sum_{x_1,s_1,x_2,s_2} u_1(x_1,s_1)u_2(x_2,s_2)\frac{\pi^\rho(s_1,s_2)}{(\#F)^2}-\chi_{{c}}({{u}_1})\chi_{{c}}({{u}_2}) \\
&=(1-\rho)\left(\sum_{x_1,x_2,s}u_1(x_1,s)u_2(x_2,s)\frac{\mu(s)}{(\#F)^2} -\chi_{{c}}({{u}_1})\chi_{{c}}({{u}_2})\right).
\end{align*}
By Lemma~\ref{lem:ucst}, we have $u_1$ and $u_2$ independent of $x$ so we get
\begin{align*}
\langle \mb{v}_1, \Sigma_\nu \mb{v}_2\rangle &=(1-\rho)\left(\sum_{x,s}u_1(x,s)u_2(x,s)\frac{\mu(s)}{\#F} -\chi_{{c}}({{u}_1})\chi_{{c}}({{u}_2})\right) \\
&=(1-\rho)\langle v_1,\Sigma_\mu v_2 \rangle = \langle v_1, \uth^{\ast} \Sigma_\mu \uth v_2 \rangle,
\end{align*}
as required.
\end{proof}

\subsection*{Acknowledgments}
The authors thank Professor Takuya Yamauchi for providing us a reference. J.B.\ is partially supported by ANR-22-CE40-0004 GoFR, ANR-24-CE40-3137 PLAGE and JSPS Invitational Fellowship L25508.
R.T.\ is partially supported by 
JSPS Grant-in-Aid for Scientific Research JP24K06711.

\bibliographystyle{plain}
%apalike
\bibliography{va}

\bigskip
\jb
\medskip
\rt

\end{document}